\documentclass[12pt]{article}
\usepackage{amsfonts, amsmath, amssymb, latexsym, eucal, amscd} 

\usepackage{amsthm}
\usepackage{amscd}

\usepackage{cite}

\newtheorem{definition}{Definition}[section]
\newtheorem{theorem}[definition]{Theorem}
\newtheorem{lemma}[definition]{Lemma}
\newtheorem{corollary}[definition]{Corollary}
\newtheorem{remark}[definition]{Remark}
\newtheorem{example}[definition]{Example}
\newtheorem{conjecture}[definition]{Conjecture}
\newtheorem{problem}[definition]{Problem}

\newtheorem{assumption}[definition]{Assumption}
\newtheorem{proposition}[definition]{Proposition}

\begin{document} 

\title{\bf Spin models and distance-regular \\  graphs
of $q$-Racah type
}
\author{
Kazumasa Nomura and 
Paul Terwilliger}
\date{}

\maketitle
\begin{abstract} Let $\Gamma$ denote a distance-regular graph, with vertex set $X$ and diameter $D\geq 3$. We assume that
$\Gamma$ is formally self-dual and $q$-Racah type. Let $A$ denote the adjacency matrix of $\Gamma$. Pick $x \in X$, and let
$A^*=A^*(x) $ denote the dual adjacency matrix of $\Gamma$ with respect to $x$. The matrices $A, A^*$ generate the 
subconstituent algebra $T=T(x)$. We assume that for every choice of $x$ the algebra $T$ contains a certain central element $Z=Z(x)$
whose significance is illuminated by the following relations:
\begin{align*}
 {\sf A} + \frac{q {\sf B}{\sf C}-q^{-1} {\sf C}{\sf B}}{q^2-q^{-2}} = Z, \qquad 
 {\sf B}+ \frac{q {\sf C}{\sf A}-q^{-1} {\sf A}{\sf C}}{q^2-q^{-2}} = Z, \qquad 
 {\sf C} + \frac{q {\sf A}{\sf B} - q^{-1} {\sf B}{\sf A}}{q^2-q^{-2}} = Z.
\end{align*}
The  matrices $\sf A$, $\sf B$ satisfy 
$ {\sf A}= (A-\varepsilon I)/\alpha$ and $ {\sf B} = (A^*-\varepsilon I)/\alpha $, where
$\alpha, \varepsilon$ are complex scalars used to describe the eigenvalues of $A$ and $A^*$.
The matrix $\sf C$ is defined using the third displayed equation.
We use $Z$ to construct a spin model $\sf W$ afforded by $\Gamma$.
We investigate the combinatorial implications of $Z$. We reverse the logical
direction and recover $Z$ from $\sf W$. We finish with some open problems.
\bigskip

 \end{abstract}
 
\section{Introduction}
\noindent
The concept of a spin model was introduced V.~F.~R Jones \cite{Jones} in his study of knot invariants.
A spin model is a symmetric matrix over the complex numbers, that satisfies two conditions called type II and type III; see Definition \ref{def:spinM} below.
\medskip

\noindent  Spin models are relevant to algebraic graph theory, as we now explain.
In \cite{Jones} Jones described three examples of spin models:
the Potts model, the square model, and the odd cyclic model.
In \cite{JaegerHS}, Fran\c{c}ois Jaeger pointed out that these 
spin models are contained in the Bose-Mesner algebra of a distance-regular
graph $\Gamma$.
For the Potts model $\Gamma$ is a complete graph,
for the square model $\Gamma$ is the $4$-cycle,
and for the odd cyclic model $\Gamma$ is an odd cycle.
Motivated by these examples, Jaeger hunted for a spin model contained in the Bose-Mesner algebra
of a strongly-regular graph $\Gamma$. He discovered such a spin model for $\Gamma$ the Higman-Sims graph \cite{HS, JaegerHS}.
Since that discovery, other authors found  spin models contained in the Bose-Mesner algebra
of  the even cycle \cite{BB},
the Hadamard graphs \cite{nomHadamard},
and the double cover of the Higman-Sims graph \cite{mun}.
\medskip

\noindent In \cite{N:algebra}  Nomura showed that a spin model $W$ is contained in a certain algebra $N(W)$, now called the Nomura algebra of $W$.
It can happen that $W \in M \subseteq N(W)$, where $M$ is the Bose-Mesner algebra of a distance-regular graph $\Gamma$. In this case, we say that
 $\Gamma$ affords $W$. Each spin model mentioned in the  previous paragraph is afforded by the distance-regular graph in question  \cite[Section~15]{nomSpinModel}.
 These examples motivate us to study the pairs $W, \Gamma$ where $W$ is a spin model and $\Gamma$ is a distance-regular graph that affords $W$.
In the next three paragraphs, we summarize what is known about these pairs. 
 \medskip
 
\noindent  Let $\Gamma$ denote a distance-regular graph that affords a spin model $W$.
In \cite[Theorem~1.1]{curtNom} Curtin and Nomura show that $\Gamma$ is formally self-dual in the sense of \cite[Section~2.3]{bcn}. Formal self-duality
 is a special case of the $Q$-polynomial property \cite{delsarte}, so $\Gamma$ must be $Q$-polynomial. Treatments of the $Q$-polynomial property can be found in \cite{bbit, bannai, bcn, dkt,int};
see also \cite[Section~9]{curtin2hom}.
 \medskip
 
\noindent  Let $\Gamma$ denote a distance-regular graph that affords a spin model $W$.
 According to  \cite{nomSM}, if $\Gamma$ is bipartite then  $\Gamma$ is 2-homogeneous in the sense of \cite{nomHomNP}.
In \cite{curtNom}, Curtin and Nomura obtain formulas for the intersection numbers of $\Gamma$.
In \cite{C:thin}, Curtin shows that every irreducible module for every subconstituent algebra of $\Gamma$ is thin in the sense of \cite[p.~381]{tSub1}.
In \cite{CW}, Caughman and Wolff describe these irreducible modules in detail.
In \cite{CNhom}, Curtin and Nomura show that $\Gamma$ is 1-homogeneous in the sense of  \cite{nomHomNP}. They also show that
with respect any vertex of $\Gamma$, the subgraph induced on the first subconstituent is strongly-regular, and the subgraph induced on
the last subconstituent is distance-regular.
\medskip

\noindent Let $\Gamma$ denote a distance-regular graph that is formally self-dual. In \cite{nomSpinModel}, we showed that if
every irreducible module for every subconstituent algebra of $\Gamma$ takes a certain form, then 
$\Gamma$  affords a spin model. We explicitly constructed this spin model for the case in which $\Gamma$ has $q$-Racah type.
\medskip
 
 \noindent We remark that the spin model concept lead to several advances beyond graph theory.
 For example, the spin model concept motivated the notions of  a spin Leonard pair \cite{C:spinLP}, a modular
 Leonard triple \cite{mlt}, and the pseudo intertwiners for a Leonard triple of $q$-Racah type \cite{pseudo}. In \cite[Theorem~8.6]{TLus} and
 \cite[Section~3]{TLus2},
an algebraic analog of a spin model is used to describe the Lusztig automorphism of the $q$-Onsager algebra.
 Given a tridiagonal pair of $q$-Racah type, an algebraic analog of a spin model is used in \cite[p.~3]{Tqtet} to turn the underlying vector space into a module for the $q$-tetrahedron algebra.
\medskip

\noindent We now describe the present paper. Let $\Gamma$ denote a distance-regular graph, with vertex set $X$ and diameter $D\geq 3$ (official definitions begin in Section 2). 
We assume that $\Gamma$ is formally self-dual and
$q$-Racah type. We also make an assumption about the subconstituent algebras of $\Gamma$, to be described shortly.
Our goal is  to construct a spin model afforded by $\Gamma$. 
 Let $A$ denote the adjacency matrix of $\Gamma$. Pick $x \in X$, and let $A^*=A^*(x)$ denote the
dual adjacency matrix of $\Gamma$ with respect to $x$. By construction, the eigenvalues of $A$ and $A^*$ can be ordered such that
\begin{align*}
 \theta_i = \theta^*_i = \alpha( a q^{2i-D}+ a^{-1} q^{D-2i}) + \varepsilon  \qquad \qquad (0 \leq i \leq D),
 \end{align*}
where $a, \alpha, \varepsilon \in \mathbb C$ and $a, \alpha $ are nonzero.
The matrices $A, A^*$ generate the subconstituent algebra $T=T(x)$.
We assume that for every choice of $x$ the algebra $T$ contains a certain central element $Z=Z(x)$, whose significance is illuminated by the following  relations:
 \begin{align*}
 {\sf A} + \frac{q {\sf B}{\sf C}-q^{-1} {\sf C}{\sf B}}{q^2-q^{-2}} &= Z, \\
 {\sf B}+ \frac{q {\sf C}{\sf A}-q^{-1} {\sf A}{\sf C}}{q^2-q^{-2}} &= Z, \\
 {\sf C} + \frac{q {\sf A}{\sf B} - q^{-1} {\sf B}{\sf A}}{q^2-q^{-2}}& = Z.
 \end{align*}
 These relations are the universal Askey-Wilson relations in $\mathbb Z_3$-symmetric form; see Remark \ref{rem:AW}.
The above matrices $\sf A$ and $\sf B$ are defined by
\begin{align*} 
 {\sf A}= \frac{A-\varepsilon I}{\alpha}, \qquad \qquad {\sf B} = \frac{A^*-\varepsilon I}{\alpha}.
 \end{align*}
The above matrix $\sf C$ is defined using the third displayed relation.
In a moment, we will use the scalars
\begin{align*}
\tau_i =  (-1)^i a^{-i} q^{i(D-i)} \qquad \qquad (0 \leq i \leq D).
\end{align*}
The Bose-Mesner algebra $M$ of $\Gamma$ is the subalgebra of ${\rm Mat}_X(\mathbb C)$ generated by $A$. This algebra has a basis
$\lbrace A_i \rbrace_{i=0}^D$ of distance matrices and a basis $\lbrace E_i \rbrace_{i=0}^D$ of primitive idempotents.
We find all the invertible matrices $W\in  M$ such that $W^{-1} {\sf B} W = {\sf C}$. As we will see, these are precisely the
matrices of the form
\begin{align*}
W = f \sum_{i=0}^D \tau_i E_i, \qquad \qquad 0 \not=f \in \mathbb C.
\end{align*}
The dual Bose-Mesner algebra $M^*=M^*(x)$ is  the subalgebra of ${\rm Mat}_X(\mathbb C)$ generated by $A^*$. This algebra has a basis
$\lbrace A^*_i \rbrace_{i=0}^D$ of dual distance matrices and a basis $\lbrace E^*_i \rbrace_{i=0}^D$ of dual primitive idempotents.
We find all the invertible matrices $W^* \in M^*$ such that $W^* {\sf A} (W^*)^{-1} ={\sf C}$. As we will see,  these are precisely the
matrices of the form
\begin{align*}
W^* = f \sum_{i=0}^D \tau_i E^*_i \qquad \qquad 0 \not= f \in \mathbb C.
\end{align*}
For the time being, pick any $0 \not=f \in \mathbb C$ and define
\begin{align*}
{\sf W} = f \sum_{i=0}^D \tau_i E_i, \qquad \qquad {\sf W}^* = f \sum_{i=0}^D \tau_i E^*_i.
\end{align*}
 Later we will pick $f$ more carefully.
Since $\sf W, W^*$ are invertible matrices in $T$, there exists an algebra isomorphism $\rho: T \to T$
that sends
\begin{align*}
S \mapsto ({\sf W^*W})^{-1} S ({\sf W^*W}) \qquad \qquad S \in T.
\end{align*}
We will show that $\rho$ sends ${\sf A}\mapsto {\sf B} \mapsto {\sf C} \mapsto  {\sf A}$. We will also show that $\rho(E_i)=E^*_i $ $(0 \leq i \leq D)$ and $\rho({\sf W})={\sf W}^*$.
From this we get the braid relation
\begin{align} \label{eq:BR}
\sf W W^* W = W^* W W^*.
\end{align}
For $0 \leq i \leq D$ let $k_i$ denote the number of vertices in $X$ at distance $i$ from $x$.
We will show that
\begin{align*}  
            \vert X \vert = \Biggl( \sum_{i=0}^D \tau_i k_i \Biggr)\Biggl( \sum_{i=0}^D \tau^{-1}_i k_i \Biggr).
\end{align*}           
We will also show that
  \begin{align*} 
 {\sf  W} = f \frac{ \sum_{i=0}^D \tau^{-1}_i A_i}{\sum_{i=0}^D \tau^{-1}_i k_i},
  \qquad \qquad
 {\sf W}^{-1} = \frac{1}{f}\, \frac{ \sum_{i=0}^D \tau_i A_i}{\sum_{i=0}^D \tau_i k_i}.
\end{align*}
These facts yield
 \begin{align} \label{eq:WcW}
         {\sf W} \circ{\sf  W}^{-1} = \vert X \vert^{-1} J,
  \end{align}
where $J\in {\rm Mat}_X(\mathbb C)$ is the all 1's matrix and $\circ $ is the entry-wise product. Using \eqref{eq:BR}, \eqref{eq:WcW} we will show that $\sf W$ is a spin model, provided that
\begin{align*}
f^2 = \vert X\vert^{1/2} \sum_{i=0}^D t^{-1}_i k_i.
\end{align*}
In this case, the spin model $\sf W$ is afforded by $\Gamma$ as we will show. 
\noindent
 Our main result is Theorem \ref{thm:main}. In the main body of the paper, we will investigate the combinatorial implications 
 of the central element $Z$. We will also reverse the logical direction, and recover $Z$ from $\sf W$. 
  At the end of the paper we give some open problems.
 \medskip
 
  \noindent The paper is organized as follows.
 Section 2 contains some preliminaries.
 In Section 3, we review the definitions of a type II matrix and spin model.
  In Section 4, we review some basic facts about a distance-regular graph $\Gamma$.
  In Section 5, we describe how the Bose-Mesner algebra $M$ and each dual Bose-Mesner algebra $M^*$ are related.
    In Section 6, we recall what it means for $\Gamma$ to be formally self-dual.
     In Section 7, we introduce the matrix $Z$.
  In Section 8, we introduce the matrices $\sf A, B, C$ and show how they are related to $Z$.
   In Section 9, we find the invertible matrices $W \in M$ such that $W^{-1} {\sf B} W = {\sf C}$, and the invertible matrices $W^* \in M^*$ such that $W^* {\sf A} (W^*)^{-1} = {\sf C}$.
    In Section 10, we introduce the matrices $\sf W, W^*$.
     In Section 11, we use $\sf W, W^*$ to obtain an algebra isomorphism $\rho : T \to T$ that sends ${\sf A} \mapsto {\sf B}\mapsto {\sf C} \mapsto {\sf A}$.
  In Section 12, we use $\rho $ to get the braid relation $\sf W W^* W = W^* W W^*$.
   In Section 13, we show that ${\sf W} \circ {\sf W^{-1}} = \vert X \vert^{-1}  J$.
    In Section 14, we show that $\sf W$ is a spin model afforded by $\Gamma$.
     In Section 15, we discuss the combinatorial implications of $Z$.
  In Section 16, we reverse the logical direction and recover $Z$ from $\sf W$.
   In Section 17, we give some suggestions for future research.
     Section 18 is an appendix that contains some formulas that are used in the main body of the paper. 
 
\section{Preliminaries}

\noindent 
We now begin our formal argument. The following concepts and notation will be used throughout the paper.
Recall the integers $\mathbb Z = \lbrace 0, \pm 1, \pm 2,\ldots \rbrace$
and the field $\mathbb C$ of  complex numbers. Let $X$ denote a nonempty finite set. An element of $X$ is called a {\it vertex}.  Let ${\rm Mat}_X(\mathbb C)$ denote the $\mathbb C$-algebra consisting of the matrices
that have rows and columns indexed by $X$ and all entries in $\mathbb C$. Let $I \in {\rm Mat}_X(\mathbb C)$ denote the identity matrix.
Let $J \in {\rm Mat}_X(\mathbb C)$ have all entries 1. For $A \in {\rm Mat}_X(\mathbb C)$ let $A^t$ denote the transpose of $A$, and let $\overline A$ denote the complex conjugate of $A$.
\medskip

\noindent For $A, B \in {\rm Mat}_X(\mathbb C)$ let $A \circ B$ denote the matrix in ${\rm Mat}_X(\mathbb C)$
that has $(y,z)$-entry $A_{y,z} B_{y,z}$ for all $y,z \in X$. We call $\circ$ the {\it Hadamard product} or {\it entrywise product}.
For $A, B \in {\rm Mat}_X(\mathbb C)$  define $\lbrack A, B \rbrack=AB-BA$.
\medskip

\noindent On occasion we will speak of the square root $\vert X \vert^{1/2}$. We always choose $\vert X \vert^{1/2}>0$.

\section{Type II matrices and spin models}

\noindent 
In this section, we first review a family of matrices in ${\rm Mat}_X(\mathbb C)$, said to have type II. We then review a special  type II matrix, called a spin model.

\begin{definition}\rm Let $W \in {\rm Mat}_X(\mathbb C)$ have all entries nonzero. Define a matrix $W^{(-)} \in {\rm Mat}_X(\mathbb C)$ that has  entries
 \begin{align*}
  W^{(-)}_{a,b} = \frac{1}{W_{b,a}} \qquad \qquad (a,b \in X).
  \end{align*}
Thus
\begin{align} \label{eq:WWJ}
W^t \circ W^{(-)} = J.
\end{align}
\end{definition}

  \begin{definition}\label{def:type2} \rm  (See \cite[Definition~2.1]{Jones}). A matrix $W\in {\rm Mat}_X(\mathbb C)$ is called {\it type II}
  whenever the following conditions hold:
  \begin{enumerate}
  \item[\rm (i)] $W$ has all entries nonzero;
  \item[\rm (ii)] $W W^{(-)} \in {\rm Span} (I)$.
   \end{enumerate}
  \end{definition}
  
  \begin{example}\rm
  Assume that $\vert X \vert=2$. Let $W \in {\rm Mat}_X(\mathbb C)$ have all entries nonzero. Write
  \begin{align*}
  W = \begin{pmatrix} r & s \\
           t & u
           \end{pmatrix}.
           \end{align*}
\noindent We have
\begin{align*}
W^{(-)} = \begin{pmatrix} r^{-1} & t^{-1} \\ s^{-1} & u^{-1}
\end{pmatrix}.
\end{align*}
Observe that
\begin{align*}
     W W^{(-)} = \begin{pmatrix} 2 & \frac{r}{t}+\frac{s}{u} \\ 
     \frac{t}{r} + \frac{u}{s} & 2 
     \end{pmatrix}.
\end{align*}
The matrix $W$ is type II if and only if $ru+st=0$. In this case,
\begin{align*}
W W^{(-)} = 2I.
\end{align*}
  \end{example}
  \noindent In the next result, we list some basic facts about  type II matrices.
  
  \begin{lemma} \label{lem:WW} Assume that $W \in {\rm Mat}_X(\mathbb C)$ is type II. Then the following {\rm (i)--(iii)} hold:
  \begin{enumerate}
  \item[\rm (i)]  $W W^{(-)} = \vert X \vert I$;
    \item[\rm (ii)] $W^{-1}$ exists;
      \item[\rm (iii)] $W^{-1} = \vert X \vert^{-1} W^{(-)}$.
  \end{enumerate}
  \end{lemma}
  \begin{proof} (i) By Definition \ref{def:type2}, there exists $\alpha \in \mathbb C$ such that $W W^{(-)} = \alpha I$. We show that $\alpha = \vert X \vert$. To do this, let $a \in X$
  and consider the $(a,a)$-entry of each side of $W W^{(-)} = \alpha I$. This yields
  \begin{align*}
  \alpha  = \sum_{b \in X} W_{a,b} W^{(-)}_{b,a} 
               = \sum_{b \in X} 1
               = \vert X\vert.
              \end{align*}
              \noindent (ii), (iii) By (i).
  \end{proof}
  
  \begin{lemma} \label{lem:type2Char}
  For $W \in {\rm Mat}_X(\mathbb C)$ the following are equivalent:
  \begin{enumerate}
  \item[\rm (i)]  $W$ is type II;
  \item[\rm (ii)] $W^{-1} $ exists and 
  \begin{align}
   W^t \circ W^{-1} = \vert X \vert^{-1} J.
   \label{eq:WCW}
   \end{align}
  \end{enumerate}
  \end{lemma}
\begin{proof}  ${\rm (i)} \Rightarrow {\rm (ii)}$ By \eqref{eq:WWJ} and Lemma \ref{lem:WW}(ii),(iii).\\
\noindent  ${\rm (ii)} \Rightarrow {\rm (i)}$  By  \eqref{eq:WCW}, the matrix $W$ has all entries nonzero and $W^{(-)}= \vert X \vert W^{-1}$. Therefore
\begin{align*}
W W^{(-)} = \vert X \vert W W^{-1} = \vert X \vert I.
\end{align*}
The matrix $W$ is type II in view of Definition \ref{def:type2}.
\end{proof}

\begin{lemma} \label{lem:MinusM}  Assume that $W \in {\rm Mat}_X(\mathbb C)$ is type II.  Let $0 \not= \alpha \in \mathbb C$. Then $\alpha W$ is type II.
\end{lemma}
\begin{proof} Use Lemma \ref{lem:type2Char}.
\end{proof}

\begin{lemma} Assume that $W \in {\rm Mat}_X(\mathbb C)$ is type II.   Then $W^t$ and $W^{-1}$ are type II.
\end{lemma}
\begin{proof} Use Lemma \ref{lem:type2Char}.
\end{proof}

 \begin{definition}\label{def:spinM} \rm (See \cite[Definition~2.1]{Jones}). A matrix $W \in {\rm Mat}_X(\mathbb C)$  is called a {\it spin model} whenever the following {\rm (i)--(iii)} hold:
 \begin{enumerate}
 \item[\rm (i)] 
  $W$ is symmetric;
  \item[\rm (ii)] $W$ is 
   type II;
 \item[\rm (iii)]   
 for all $a,b,c \in X$,
  \begin{align} \label{eq:type3}
  \sum_{e\in X} \frac{W_{e,b} W_{e,c}}{W_{e,a}} = \vert X \vert^{1/2} \frac{ W_{b,c}}{W_{a,b}W_{c,a}}.
  \end{align}
  \end{enumerate}
  We call condition \eqref{eq:type3} the {\it type III} condition or the {\it star-triangle} condition.
  \end{definition}
\begin{lemma} Assume that $W \in {\rm Mat}_X(\mathbb C)$ is a spin model. Then $-W$ is a spin model.
\end{lemma}
\begin{proof} The matrix $-W$ is symmetric. The matrix $-W$ is type II by Lemma \ref{lem:MinusM}. The equation \eqref{eq:type3} remains valid if we replace $W$ by $-W$.
\end{proof}
\begin{lemma} \label{lem:WminusSM} {\rm (See \cite[Section~2]{Jones} and \cite[Proposition~2.1]{kawagoe}).}  Assume that $W \in {\rm Mat}_X(\mathbb C)$ is a spin model. Then $W^{(-)}$ is a spin model.
\end{lemma}

 \section{Distance-regular graphs}
 
 \noindent In this section we review some basic results about
 distance-regular graphs, focussing on the $Q$-polynomial property and the subconstituent algebra. Our review is adapted from
 \cite{tSub1,tSub2, tSub3} 
 and 
 \cite{int};  for more information see
 \cite{bannai, bbit, bcn, dkt}.
Let $\Gamma = (X, \mathcal R)$ denote a finite, undirected, connected graph,
without loops or multiple edges, with vertex set $X$ and 
adjacency relation
$\mathcal R$.  
Let $\partial $ denote the
path-length distance function for $\Gamma $,  and define
$D = \mbox{max}\lbrace \partial(y,z) \vert y,z \in X\rbrace $.  
We call $D$  the {\it diameter} of $\Gamma $. For all vertices $y \in X$ and all integers $i$ $(0 \leq i \leq D)$ define the set
$\Gamma_i(y) = \lbrace z \in X\vert \partial(y,z)=i \rbrace$. 
The graph  $\Gamma$ is called {\it distance-regular}
whenever for all integers $h,i,j\;(0 \leq h,i,j \leq D)$ 
and  all
vertices $y,z \in X$ with $\partial(y,z)=h,$ the number
$p^h_{i,j} = \vert \Gamma_i(y) \cap \Gamma_j(z) \vert$
is independent of $y$ and $z$. The $p^h_{i,j}$  are called
the {\it intersection numbers} of $\Gamma$. For the rest of this paper, we assume that $\Gamma$ is distance-regular with $D\geq 3$.
By construction $p^h_{i,j} = p^h_{j,i} $ for $0 \leq h,i,j\leq D$.  
By the triangle inequality, the following hold for $0 \leq h,i,j\leq D$:
\begin{enumerate}
\item[\rm (i)] $p^h_{i,j}= 0$ if one of $h,i,j$ is greater than the sum of the other two;
\item[\rm (ii)] $p^h_{i,j}\not=0$ if one of $h,i,j$ is equal to the sum of the other two.
\end{enumerate}
We abbreviate
\begin{align*}
c_i = p^i_{1,i-1} \;\; (1 \leq i \leq D), \qquad a_i = p^i_{1,i} \;\; (0 \leq i \leq D), \qquad  b_i = p^i_{1,i+1} \;\; (0 \leq i \leq D-1).
\end{align*}
We emphasize that $c_i \not= 0 $  $(1 \leq i \leq D)$ and $b_i \not= 0 $ $(0 \leq i \leq D-1)$.
Note that $a_0=0$ and $c_1=1$.
For $0 \leq i \leq D$ define $k_i = p^0_{i,i}$ and note that $k_i = \vert \Gamma_i(y) \vert$ for all $y \in X$.
We call $k_i$ the $i$th {\it valency} of $\Gamma$. 
By \cite[p.~195]{bannai} we have
\begin{align} \label{eq:kicibi}
k_i = \frac{ b_0 b_1 \cdots b_{i-1}}{c_1 c_2 \cdots c_i} \qquad \qquad (0 \leq i \leq D).
\end{align}
Abbreviating $k=k_1$ we have
\begin{align*}
k=c_i + a_i + b_i  \qquad \qquad (0 \leq i \leq D),
\end{align*}
where $c_0=0$ and $b_D=0$.
\medskip

\noindent We recall the Bose-Mesner algebra of $\Gamma.$ 
For 
$0 \leq i \leq D$ define $A_i \in {\rm Mat}_X(\mathbb C)$ that has
$(y,z)$-entry
\begin{align*}
(A_i)_{y,z} = \begin{cases}  
1, & {\mbox{\rm if $\partial(y,z)=i$}};\\
0, & {\mbox{\rm if $\partial(y,z) \ne i$}}
\end{cases}
 \qquad \qquad (y,z \in X).
\end{align*}
We call $A_i$ the $i$th {\it distance matrix} of $\Gamma$. 
We call $A=A_1$  the {\it adjacency
matrix} of $\Gamma$. Note that
(i) $A_0 = I$;
 (ii)
$J=\sum_{i=0}^D A_i $;
(iii)  $A_i^t = A_i  \;(0 \leq i \leq D)$;
(iv)  $\overline{A_i} = A_i  \;(0 \leq i \leq D)$;
(v) $A_iA_j = \sum_{h=0}^D p^h_{i,j} A_h \;( 0 \leq i,j \leq D) $.
Consequently the matrices
 $\lbrace A_i\rbrace_{i=0}^D$
form a basis for a commutative subalgebra $M$ of ${\rm Mat}_X(\mathbb C)$, called the 
{\it Bose-Mesner algebra} of $\Gamma$.
The matrix $A$ generates $M$ \cite[Corollary~3.4]{int}. 
The matrices $\lbrace A_i \rbrace_{i=0}^D$ are symmetric and mutually commute, so they can be simultaneously diagonalized over the real numbers. Consequently
$M$ has a second basis 
$\lbrace E_i\rbrace_{i=0}^D$ such that
(i) $E_0 = |X|^{-1}J$;
(ii) $I=\sum_{i=0}^D E_i$;
(iii) $E_i^t =E_i  \;(0 \leq i \leq D)$;
(iv) $\overline{E_i} =E_i  \;(0 \leq i \leq D)$;
(v) $E_iE_j =\delta_{i,j}E_i  \;(0 \leq i,j \leq D)$.
We call $\lbrace E_i\rbrace_{i=0}^D$  the {\it primitive idempotents}
of $\Gamma$. 
\medskip

\noindent We recall the first and second eigenmatrices of $\Gamma$. Let ${\rm Mat}_{D+1}(\mathbb C)$ denote the $\mathbb C$-algebra consisting of the 
$(D+1) \times (D+1)$ matrices that have all entries in $\mathbb C$. We index the rows and columns by $0,1,\ldots, D$. Since $\lbrace A_i \rbrace_{i=0}^D$ and
$\lbrace E_i \rbrace_{i=0}^D$ are bases for $M$, there exist matrices $P, Q \in {\rm Mat}_{D+1}(\mathbb C)$ such that
\begin{align}  \label{eq:APQ}
A_j = \sum_{i=0}^D P_{i,j} E_i, \qquad \qquad E_j = \vert X \vert^{-1} \sum_{i=0}^D Q_{i,j} A_i \qquad \qquad (0 \leq j \leq D).
\end{align}
Note that $PQ = QP = \vert X \vert I$. By construction, the entries of $P$ and $Q$ are real.
We call $P$ (resp. $Q$) the {\it first eigenmatrix} (resp. {\it second eigenmatrix}) of $\Gamma$; see \cite[p.~60]{bannai}.
 \medskip
 
 \noindent We recall the Krein parameters of $\Gamma$. 
Note that $A_i \circ A_j = \delta_{i,j} A_i$  $(0 \leq i,j\leq D)$. Therefore  $M$  is closed under $\circ$.
Consequently, there exist scalars $q^{h}_{i,j} \in \mathbb C $ $(0 \leq h,i,j\leq D)$ such that
\begin{align*} 
E_i \circ E_j = \vert X \vert^{-1} \sum_{h=0}^D q^h_{i,j} E_h \qquad \qquad (0 \leq i,j\leq D).
\end{align*}  
The scalars $q^{h}_{i,j}$ are called the {\it Krein parameters} of $\Gamma$.
By construction $q^h_{i,j} = q^h_{j,i}$  $(0 \leq h,i,j\leq D)$.
By \cite[p.~69]{bannai} the scalar $q^h_{i,j}$ is real and nonnegative $(0 \leq h,i,j\leq D)$.
\medskip

 \noindent We recall the $Q$-polynomial property. The ordering $\lbrace E_i \rbrace_{i=0}^D$ is said to be {\it $Q$-polynomial} whenever the following hold
for $0 \leq h,i,j\leq D$:
\begin{enumerate}
\item[\rm (i)] $q^h_{i,j}=0$ if one of $h,i,j$ is greater than the sum of the other two;
\item[\rm (ii)]  $q^h_{i,j}\not=0$ if one of $h,i,j$ is equal to the sum of the other two.
\end{enumerate}
For the rest of this section, we  assume that the ordering $\lbrace E_i \rbrace_{i=0}^D$ is $Q$-polynomial. We abbreviate
\begin{align*}
c^*_i = q^i_{1,i-1} \;\; (1 \leq i \leq D), \qquad a^*_i = q^i_{1,i} \;\; (0 \leq i \leq D), \qquad  b^*_i = q^i_{1,i+1} \;\; (0 \leq i \leq D-1).
\end{align*}
We emphasize that $c^*_i \not= 0 $  $(1 \leq i \leq D)$ and $b^*_i \not= 0 $ $(0 \leq i \leq D-1)$. By \cite[Lemma~5.15]{int} we have $a^*_0=0$ and $c^*_1=1$.
For $0 \leq i \leq D$ define $k^*_i = q^0_{i,i}$. By \cite[Lemma~5.15]{int} we have $k^*_i = {\rm rank}(E_i)$.
We call $k^*_i$ the $i$th {\it multiplicity} of the ordering $\lbrace E_j \rbrace_{j=0}^D$. 
By \cite[p.~276]{bannai} we have
\begin{align*}
k^*_i = \frac{ b^*_0 b^*_1 \cdots b^*_{i-1}}{c^*_1 c^*_2 \cdots c^*_i} \qquad \qquad (0 \leq i \leq D).
\end{align*}
 Abbreviate $k^*=k^*_1$. By \cite[Lemma~5.15]{int} we have
 \begin{align*}
 k^* = c^*_i + a^*_i + b^*_i \qquad \qquad (0 \leq i \leq D),
 \end{align*}
 where $c^*_0=0$ and $b^*_D=0$.
 \medskip

 \noindent
We  recall the dual Bose-Mesner algebras of $\Gamma$.
For the rest of this section, fix
a vertex $x \in X$. We call $x$ the {\it base vertex}.
For 
$ 0 \leq i \leq D$ let $E_i^*=E_i^*(x)$ denote the diagonal
matrix in 
 ${\rm Mat}_X(\mathbb C)$
 that has $(y,y)$-entry
\begin{align*}
(E_i^*)_{y,y} = \begin{cases} 1, & \mbox{\rm if $\partial(x,y)=i$};\\
0, & \mbox{\rm if $\partial(x,y) \ne i$}
\end{cases}
 \qquad \qquad (y \in X).
\end{align*}
We call $E_i^*$ the  $i$th {\it dual primitive idempotent of $\Gamma$
 with respect to $x$} \cite[p.~378]{tSub1}. Note that
(i) $I=\sum_{i=0}^D E_i^*$;
(ii) $(E_i^*)^t = E_i^*$ $(0 \leq i \leq D)$;
(iii) $\overline{E^*_i} = E_i^*$ $(0 \leq i \leq D)$;
(iv) $E_i^*E_j^* = \delta_{i,j}E_i^* $ $(0 \leq i,j \leq D)$.
Consequently the matrices
$\lbrace E_i^*\rbrace_{i=0}^D$ form a 
basis for a commutative subalgebra
$M^*=M^*(x)$ of 
${\rm Mat}_X(\mathbb C)$.
We call 
$M^*$ the {\it dual Bose-Mesner algebra of
$\Gamma$ with respect to $x$} \cite[p.~378]{tSub1}.
\medskip

\noindent We recall the dual distance matrices of $\Gamma$.
  For $0 \leq i \leq D$ let  $A^*_i = A^*_i(x) $ denote the diagonal matrix in ${\rm Mat}_X(\mathbb C)$
that has $(y,y)$-entry
\begin{align*}
 (A^*_i)_{y,y} = \vert X \vert (E_i)_{x,y} \qquad \qquad (y \in X).
 \end{align*}
 By \cite[Lemma~5.8]{int} the matrices $\lbrace A^*_i \rbrace_{i=0}^D$ form a basis for $M^*$. We have
 (i) $A^*_0=I$;
(ii) $\sum_{i=0}^D A^*_i = \vert X \vert E^*_0$;
(iii)  $(A^*_i)^t = A^*_i \;(0 \leq i \leq D)$;
(iv)  $\overline{A^*_i}= A^*_i \;(0 \leq i \leq D)$;
(v) $A^*_i A^*_j = \sum_{h=0}^D q^h_{i,j} A^*_h\; (0 \leq i,j\leq D)$.
We call $A^*_i$ the $i$th {\it dual distance matrix of $\Gamma$ with respect to $x$ and the ordering $\lbrace E_j\rbrace_{j=0}^D$}. We call $A^*=A^*_1$ 
the {\it dual adjacency matrix of $\Gamma$ with respect to $x$ and the ordering $\lbrace E_j\rbrace_{j=0}^D$}. The matrix $A^*$ generates $M^*$ \cite[Corollary~11.6]{int}.
By \cite[Lemmas~3.9,~5.10]{int} we have
\begin{align} \label{eq:AsPQ}
A^*_j = \sum_{i=0}^D Q_{i,j} E^*_i, \qquad \qquad E^*_j = \vert X \vert^{-1} \sum_{i=0}^D P_{i,j} A^*_i \qquad \qquad (0 \leq j \leq D).
\end{align}

\noindent
We recall the subconstituent algebras of $\Gamma$.
Let $T=T(x)$ denote the subalgebra of ${\rm Mat}_X(\mathbb C)$ generated by 
$M$ and $M^*$. 
The algebra $T$ is finite-dimensional and noncommutative. We call $T$ the {\it subconstituent algebra}
(or {\it Terwilliger algebra}) {\it of $\Gamma$ 
 with respect to $x$} \cite[Definition 3.3]{tSub1}. Note that $T$ is generated by $A, A^*$.
\medskip

\noindent We recall the scalars $\lbrace \theta_i \rbrace_{i=0}^D$ and  $\lbrace \theta^*_i \rbrace_{i=0}^D$.
Abbreviate $\theta_i = P_{i,1}$ and $\theta^*_i = Q_{i,1}$ for $0 \leq i \leq D$. By \eqref{eq:APQ} we have
\begin{align*}
A = \sum_{i=0}^D \theta_i E_i, \qquad \qquad E_1 = \vert X \vert^{-1} \sum_{i=0}^D \theta^*_i A_i.
\end{align*}
By \eqref{eq:AsPQ} we have
\begin{align*}
A^* = \sum_{i=0}^D \theta^*_i E^*_i, \qquad \qquad E^*_1 = \vert X \vert^{-1} \sum_{i=0}^D \theta_i A^*_i.
\end{align*}
For $0 \leq i \leq D$ the scalar $\theta_i$ is the eigenvalue of $A$ associated with $E_i$, and $\theta^*_i$ is the eigenvalue of $A^*$ associated with $E^*_i$.
The scalars $\lbrace \theta_i \rbrace_{i=0}^D$ are mutually distinct since $A$ generates $M$ \cite[Lemma~3.5]{int}.
The scalars $\lbrace \theta^*_i \rbrace_{i=0}^D$ are mutually distinct since $A^*$ generates $M^*$ \cite[Lemma~11.7]{int}.
\medskip

 \section{How $M$ and $M^*$ are related}
 
   \noindent We continue to discuss the distance-regular graph $\Gamma=(X,\mathcal R)$ with diameter $D\geq 3$. Throughout this section, $\lbrace E_i \rbrace_{i=0}^D$ denotes a $Q$-polynomial
   ordering of the primitive idempotents of $\Gamma$.  Throughout this section, we fix $x \in X$ and write $T=T(x)$.
 We describe how the Bose-Mesner algebra $M$ and the dual Bose-Mesner algebra $M^*$ are related. We also discuss how the eigenvalues
 $\lbrace \theta_i \rbrace_{i=0}^D$, $\lbrace \theta^*_i\rbrace_{i=0}^D$ are related. We consider a special case called $q$-Racah.
 
 \begin{lemma} \label{lem:TPR} {\rm (See \cite[Lemma~3.2]{tSub1}).}
For $0 \leq h,i,j\leq D$ we have
\begin{enumerate}
\item[\rm (i)] $E^*_h A_i E^*_j = 0$ if and only if  $p^h_{i,j}=0$;
\item[\rm (ii)] $E_h A^*_i E_j = 0 $ if and only if  $q^h_{i,j}=0$.
\end{enumerate}
 \end{lemma}
 
\begin{lemma} \label{lem:linDep} {\rm (See 
\cite[Section~7]{sum2}).}
\begin{enumerate}
\item[\rm (i)] the following matrices are linearly independent:
\begin{align*} 
\lbrace E^*_hA_i E^*_j \vert 0 \leq h,i,j\leq D, \; p^h_{i,j}\not=0\rbrace;
\end{align*}
\item[\rm (ii)]  the following matrices are linearly independent:
\begin{align*} 
\lbrace E_h A^*_i E_j \vert 0 \leq h,i,j\leq D, \; q^h_{i,j}\not=0\rbrace.
\end{align*}
\end{enumerate}
\end{lemma}

\begin{lemma} \label{lem:red1} {\rm (See \cite[Lemma~6.1]{int}).} For $0 \leq i \leq D$ we have
\begin{align*}
&E_0 E^*_0 A_i = E_0 E^*_i, \qquad \qquad E^*_0 E_0 E^*_i = \vert X \vert^{-1} E^*_0 A_i,
\\
& A_i E^*_0 E_0 = E^*_i E_0, \qquad \qquad E^*_i E_0 E^*_0 = \vert X \vert^{-1} A_i E^*_0.
\end{align*}
\end{lemma}

\begin{lemma} \label{lem:red2} {\rm (See \cite[Lemma~6.2]{int}).}  For $0 \leq i \leq D$ we have
\begin{align*}
&E^*_0 E_0 A^*_i = E^*_0 E_i, \qquad \qquad E_0 E^*_0 E_i = \vert X \vert^{-1} E_0 A^*_i,
\\
& A^*_i E_0 E^*_0 = E_i E^*_0, \qquad \qquad E_i E^*_0 E_0 = \vert X \vert^{-1} A^*_i E_0.
\end{align*}
\end{lemma}

\begin{lemma}
\label{lem:r5} {\rm (See \cite[Lemma~6.7]{int}).} 
For $0 \leq i \leq D$ we have
\begin{align*}
E_0 E^*_i E_0 = \vert X \vert^{-1} k_i E_0, \qquad \qquad
E^*_0 E_i E^*_0 = \vert X \vert^{-1} k^*_i E^*_0.
\end{align*}
\end{lemma}
 
 \begin{lemma}
\label{lem:r6} {\rm (See \cite[Section~2]{int}).} 
We have
\begin{align*}
E_0 E^*_0 E_0 = \vert X \vert^{-1} E_0, \qquad \qquad 
E^*_0 E_0 E^*_0 = \vert X \vert^{-1} E^*_0.
\end{align*}
\end{lemma}

\begin{lemma} \label{lem:need} {\rm (See \cite[Lemma~6.13]{int}).} 
For $S \in M$ we have the logical implications
\begin{align*}
S=0\quad  \Leftrightarrow\quad SE^*_0=0 \quad \Leftrightarrow \quad E^*_0S=0.
\end{align*}
Moreover for $S^* \in M^*$ we have the logical implications
\begin{align*}
S^*=0 \quad \Leftrightarrow \quad S^*E_0=0 \quad \Leftrightarrow\quad  E_0S^*=0.
\end{align*}
\end{lemma}

\noindent Next, we describe how $A$ and $A^*$ are related.
\begin{lemma} \label{thm:LEMOne} {\rm (See \cite[Lemma~5.4]{tSub3}).} 
There exists a unique sequence of complex scalars $\beta$, $\gamma$, $\gamma^*$, $\varrho$, $\varrho^*$
such that
\begin{align*}
0 &= \lbrack A, A^2 A^* - \beta A A^* A + A^* A^2 - \gamma(AA^*+ A^*A) - \varrho A^* \rbrack,
\\
0 &= \lbrack A^*, A^{*2} A - \beta A^* A A^* + A A^{*2} - \gamma^*(A^*A+ AA^*) - \varrho^* A \rbrack.
\end{align*}
\end{lemma}
\noindent The above two relations are called the {\it  tridiagonal relations}; see \cite{qSerre}.

\begin{lemma} \label{prop:five}  {\rm (See \cite[Lemma~5.4]{tSub3})}. We refer to the scalars 
 $\beta, \gamma, \gamma^*,
\varrho, \varrho^*$ from Lemma \ref{thm:LEMOne}.
\begin{enumerate}
\item[\rm (i)] 
The expressions
\begin{align*}
\frac{\theta_{i-2}-\theta_{i+1}}{\theta_{i-1}-\theta_i},\qquad \qquad  
 \frac{\theta^*_{i-2}-\theta^*_{i+1}}{\theta^*_{i-1}-\theta^*_i} 
\end{align*} 
 are both equal to $\beta +1$ for $2\leq i \leq D-1$;
 \item[\rm (ii)]   
$\gamma = \theta_{i-1}-\beta \theta_i + \theta_{i+1} $ $(1 \leq i \leq D-1)$;
\item[\rm (iii)] 
$\gamma^* = \theta^*_{i-1}-\beta \theta^*_i + \theta^*_{i+1}$ $(1 \leq i \leq D-1)$;
\item[\rm (iv)] 
$\varrho = \theta^2_{i-1}-\beta \theta_{i-1}\theta_i+\theta_i^2-\gamma (\theta_{i-1}+\theta_i)$ $(1 \leq i \leq D)$;
\item[\rm (v)] 
$\varrho^*= \theta^{*2}_{i-1}-\beta \theta^*_{i-1}\theta^*_i+\theta_i^{*2}-
\gamma^* (\theta^*_{i-1}+\theta^*_i)$ $(1 \leq i \leq D)$.
\end{enumerate}
\end{lemma}

\noindent The scalars $\lbrace \theta_i \rbrace_{i=0}^D$ and  $\lbrace \theta^*_i \rbrace_{i=0}^D$ can be expressed in closed form \cite[Theorem~11.2]{someAlg}. There are several cases.  The
``most general'' case is called $q$-Racah, and described below. See \cite{bbit, bannai, cerzo, qRacTet, augIto, someAlg, nomSpinModel, TZ} for more information about the $q$-Racah case.

\begin{definition} \label{def:qRacah} \rm The $Q$-polynomial ordering $\lbrace E_i \rbrace_{i=0}^D$ is said to have {\it $q$-Racah type} whenever
there exist $q, u, v, \varepsilon,u^*,v^*,\varepsilon^* \in \mathbb C$
\noindent such that
\begin{align*}
& q \not=0,\quad q^2 \not=1, \quad q^2 \not=-1, \quad uvu^*v^* \not=0, \\
&\theta_i =  u q^{2i-D} + v q^{D-2i}+ \varepsilon  \qquad \qquad  (0 \leq i \leq D), \\
&\theta^*_i =  u^* q^{2i-D} + v^* q^{D-2i}+\varepsilon^* \qquad \quad (0 \leq i \leq D).
\end{align*}
\end{definition}

\noindent Assume that the ordering $\lbrace E_i \rbrace_{i=0}^D$ has $q$-Racah type. For convenience, we adjust the notation in Definition \ref{def:qRacah}.
Define scalars $a, b \in \mathbb C$ such that
\begin{align}
a^2 = u/v, \qquad \qquad b^2 = u^*/v^*.       \label{eq:ab}
\end{align}
Note that $a,b$ are nonzero. We have
\begin{align} \label{eq:th1}
 \theta_i =  \alpha( a q^{2i-D}+ a^{-1} q^{D-2i}) + \varepsilon  \qquad \qquad (0 \leq i \leq D), \\
  \theta^*_i =  \alpha^*( b q^{2i-D}+ b^{-1} q^{D-2i}) + \varepsilon^*  \qquad \qquad (0 \leq i \leq D), \label{eq:th2}
\end{align}
where $\alpha = a v$ and $\alpha^* = b v^*$. Note that $\alpha^2=uv$ and $(\alpha^*)^2=u^*v^*$.

\begin{example}\rm \label{ex:cycle}
Assume that $\Gamma$ is an ordinary $N$-cycle. Note that $N=2D$ or $N=2D+1$. 
Pick $q \in \mathbb C$ such that $q^2$ is a primitive $N$th root of unity.  By \cite[p.~304]{bannai} there exists a $Q$-polynomial ordering $\lbrace E_i \rbrace_{i=0}^D$ of the primitive idempotents of
$\Gamma$ such that
\begin{align*}
 \theta_i = \theta^*_i = q^{2i} + q^{-2i} \qquad \qquad (0 \leq i \leq D).
 \end{align*}
 The ordering $\lbrace E_i \rbrace_{i=0}^D$ satisfies the conditions of Definition  \ref{def:qRacah}, with 
 \begin{align*}
 u = u^*= q^D, \qquad \qquad 
 v = v^*= q^{-D}, \qquad \qquad 
 \varepsilon = \varepsilon^*=0.
 \end{align*}
 Therefore,  this ordering has $q$-Racah type. The corresponding scalars $a, b, \alpha, \alpha^*$ from  \eqref{eq:ab}--\eqref{eq:th2} satisfy
 \begin{align*}
 a^2 = b^2 = q^{2D}, \qquad \qquad 
 \alpha  = a q^{-D}, \qquad \qquad  \alpha^*= b q^{-D}.
 \end{align*}
 Note that $\alpha^2=1$ and $(\alpha^*)^2 = 1$.
 \end{example}

 \section{Formal self-duality}
 \noindent We continue to discuss the distance-regular graph $\Gamma=(X,\mathcal R)$ with diameter $D\geq 3$. Throughout this section, $\lbrace E_i \rbrace_{i=0}^D$ denotes an ordering of the primitive
 idempotents of $\Gamma$.
 
 \begin{definition}\label{def:fsd} \rm (See \cite[Section~2.3]{bcn}).
The ordering $\lbrace E_i \rbrace_{i=0}^D$ is said to be {\it formally self-dual} whenever
$P=Q$.
\end{definition}

\begin{lemma} {\rm (See \cite[Section~2.3]{bcn}).} Assume that the ordering $\lbrace E_i \rbrace_{i=0}^D$ is formally self-dual.
Then $p^h_{i,j} = q^h_{i,j} $ for $0 \leq h,i,j\leq D$. 
\end{lemma}

\begin{lemma} \label{lem:ntspin} {\rm (See \cite[Proposition~10.4]{nomSpinModel}).} The following are equivalent:
\begin{enumerate}
\item[\rm (i)] the ordering $\lbrace E_i \rbrace_{i=0}^D$ is formally self-dual;
\item[\rm (ii)] the ordering $\lbrace E_i \rbrace_{i=0}^D$ is $Q$-polynomial, and $\theta_i = \theta^*_i$ for $0 \leq i \leq D$.
\end{enumerate}
\end{lemma}

\begin{example}\rm \label{ex:cycle2}
For the $N$-cycle  $\Gamma$ 
in Example \ref{ex:cycle}, the given ordering $\lbrace E_i \rbrace_{i=0}^D$ of the primitive idempotents is formally self-dual by Lemma
\ref{lem:ntspin}.
\end{example}

\noindent Sometimes we  speak of $\Gamma$ being formally self-dual and $q$-Racah type. This means that there exists an ordering
of the primitive idempotents of $\Gamma$ that is formally self-dual and $q$-Racah type.

   \section{The matrix $Z$}
   
 \noindent We continue to discuss the distance-regular graph $\Gamma=(X, \mathcal R)$ with diameter $D\geq 3$. 
   We fix $x \in X$ and write $T=T(x)$. We make two assumptions about $\Gamma$, one of which involves $x$.
   We use these assumptions  to construct a central element $Z=Z(x)$ in $T$.
 
 \begin{assumption} \label{def:spintype} \rm Let $\Gamma=(X,\mathcal R)$ denote a distance-regular graph with diameter $D\geq 3$. 
 We assume that there exists an ordering $\lbrace E_i \rbrace_{i=0}^D$ of the primitive idempotents of $\Gamma$ that is formally self-dual and $q$-Racah type.
 By Lemma  \ref{lem:ntspin} and the discussion below Definition \ref{def:qRacah},
 there exist nonzero $a, \alpha \in \mathbb C$ and $\varepsilon \in \mathbb C$ such that
 \begin{align}
 \theta_i = \theta^*_i = \alpha( a q^{2i-D}+ a^{-1} q^{D-2i}) + \varepsilon  \qquad \qquad (0 \leq i \leq D).  \label{eq:Gtheta}
 \end{align}
 \end{assumption}
 
  \begin{assumption} \label{def:spintypex} \rm We refer to Assumption \ref{def:spintype}. We fix $x \in X$ and write $T=T(x)$.
 We assume that the adjacency matrix $A$ and dual adjacency matrix $A^*$ satisfy  
 \begin{align}
 \label{eq:MainAssume}
 \begin{split}
& \sum_{i=0}^D E^*_i \frac{A -\varepsilon I}{\alpha} E^*_i  \biggl(1+ \frac{\theta_{i}-\varepsilon}{\alpha}\,\frac{1}{q+q^{-1}}\biggr) \\
& =
  \sum_{i=0}^D E_i \frac{A^* -\varepsilon I}{\alpha} E_i  \biggl(1+ \frac{\theta_{i}-\varepsilon}{\alpha}\,\frac{1}{q+q^{-1}}\biggr).
  \end{split}
\end{align}
 \end{assumption}

 \begin{example}\label{ex:cycle3} \rm We refer to the $N$-cycle  $\Gamma$ from Example \ref{ex:cycle}. Recall the scalar $q$ and the ordering $\lbrace E_i\rbrace_{i=0}^D$ such that
 \begin{align*}
 \theta_i = \theta^*_i = q^{2i} + q^{-2i} \qquad \qquad (0 \leq i \leq D).
 \end{align*}
 By Examples \ref{ex:cycle}, \ref{ex:cycle2}  the ordering $\lbrace E_i \rbrace_{i=0}^D$ is formally self-dual and $q$-Racah type. 
If $N=2D$  then define $a \in \mathbb C$ such that $a^2=-1$. If $N=2D+1$ then define $a=-q^{-D-1}$. For either parity of $N$, define $\alpha=a q^{-D}$ and $\varepsilon=0$.
Observe that $a, \alpha, \varepsilon $ satisfy  \eqref{eq:Gtheta}.
 We claim that $a, \alpha, \varepsilon$  satisfy  \eqref{eq:MainAssume}.
 To establish the claim, we show that each side of  \eqref{eq:MainAssume} is equal to 0. First assume that $N=2D$. Then $a_i = a^*_i = 0$ for 
$0 \leq i \leq D$. By this and Lemma \ref{lem:TPR}, 
\begin{align*}
E^*_i A E^*_i = 0, \qquad \qquad E_i A^* E_i = 0, \qquad \qquad (0 \leq i \leq D).
\end{align*}
Therefore, each side of \eqref{eq:MainAssume} is equal to 0.
Next assume that $N=2D+1$. Then $a_i = a^*_i = 0 $ $(0 \leq i \leq D-1)$ and $a_D = a^*_D = 1$. By Lemma \ref{lem:TPR},
\begin{align*}
&E^*_i A E^*_i = 0, \qquad \qquad E_i A^* E_i = 0, \qquad \qquad (0 \leq i \leq D-1).
\end{align*}
One checks that
\begin{align*}
&1+ \frac{\theta_{D}-\varepsilon}{\alpha}\,\frac{1}{q+q^{-1}}=0.
\end{align*}
By these comments, each side of \eqref{eq:MainAssume} is equal to 0.
We have shown that  $a, \alpha, \varepsilon$ satisfy  \eqref{eq:MainAssume}.
 \end{example}

 \begin{remark}\label{rem:notation} \rm The notation $a, \alpha, q$  in \eqref{eq:Gtheta},  \eqref{eq:MainAssume} 
 matches the notation in \cite{nomSpinModel}.
 The scalar $\varepsilon $ in  \eqref{eq:Gtheta},  \eqref{eq:MainAssume} 
 is called  $\beta$ in \cite{nomSpinModel}.
 The  papers \cite{CW, curtNom, CNhom} use a slightly different notation. 
 The $q$ in \cite{CW} is the same as our $q^2$. The $\eta$ in \cite{CW} is the same as our  $-aq^{1-D}$.
 The $p$ in  \cite{curtNom, CNhom}
 is the same as our $q^2$. The $x$ in  \cite{curtNom, CNhom} is the same as our $-aq^{1-D}$.
  \end{remark}

\noindent   Assumption \ref{def:spintype} is in effect from now until the end of Section 15.    Assumption \ref{def:spintypex} is in effect from now until the end of Section 13. 
 \medskip

\begin{lemma} \label{lem:gam} The scalars $\beta, \gamma, \gamma^*, \varrho, \varrho^*$ from Lemma    \ref{thm:LEMOne}
satisfy
\begin{align*}
\beta &= q^2+q^{-2}, \\
\gamma &= \gamma^* = - \varepsilon (q-q^{-1})^2 = \varepsilon(2-\beta), \\
 \varrho &= \varrho^* = (q-q^{-1})^2 \varepsilon^2 - (q^2
-q^{-2})^2 \alpha^2 = (\beta-2) \varepsilon^2 + (4-\beta^2) \alpha^2.
\end{align*}
\end{lemma} 
\begin{proof} Use Lemma \ref{prop:five} and  \eqref{eq:Gtheta}.
\end{proof}
 
 \begin{definition}\label{def:Z} \rm  
 Let the matrix $Z=Z(x)$ denote the common value of \eqref{eq:MainAssume}.
 \end{definition}
 
 \noindent As we discuss $Z$, the following abbreviations will be convenient.
  Define
  \begin{align}
 \vartheta_i = a q^{2i-D} + a^{-1} q^{D-2i} \qquad \qquad i \in \mathbb Z.  \label{eq:vtheta}
 \end{align}
 By \eqref{eq:Gtheta},
 \begin{align*}
 \theta_i =\theta^*_i = \alpha \vartheta_i + \varepsilon \qquad \qquad (0 \leq i \leq D).
 \end{align*}
 Note that $\lbrace \vartheta_i \rbrace_{i=0}^D$ are mutually distinct because $\lbrace \theta_i \rbrace_{i=0}^D$ are
 mutually distinct.
Define
 \begin{align} \label{eq:AB}
 {\sf A}= \frac{A-\varepsilon I}{\alpha}, \qquad \qquad {\sf B} = \frac{A^*-\varepsilon I}{\alpha}.
 \end{align}
 By construction,
 \begin{align} \label{eq:AsumBsum}
 {\sf A} = \sum_{i=0}^D \vartheta_i E_i, \qquad \qquad {\sf B} = \sum_{i=0}^D \vartheta_i E^*_i.
 \end{align}
 By \eqref{eq:MainAssume}, \eqref{eq:AB}, and the definition of $Z$.
  \begin{align}
& \sum_{i=0}^D E^*_i {\sf A} E^*_i  \biggl(1+ \frac{\vartheta_i }{q+q^{-1}}\biggr)  = Z =
  \sum_{i=0}^D E_i {\sf B} E_i  \biggl(1+ \frac{\vartheta_i }{q+q^{-1}}\biggr).         \label{eq:MainAssumeBF}
\end{align}

\noindent Recall that $\lbrace E_i\rbrace_{i=0}^D$ is a basis for the Bose-Mesner algebra $M$, and 
 $\lbrace E^*_i\rbrace_{i=0}^D$ is a basis for the dual Bose-Mesner algebra $M^*=M^*(x)$.
Recall that the algebra $T$ is generated by $M$ and $M^*$.
 \begin{proposition} \label{lem:Z1} The following {\rm (i)--(iv)} hold:
 \begin{enumerate}
 \item[\rm (i)] $Z \in T$;
 \item[\rm (ii)] for $0 \leq i \leq D$,
 \begin{align*}
  Z E_i = E_i Z = E_i {\sf B} E_i  \biggl( 1 + \frac{\vartheta_i}{q+q^{-1}} \biggr);
  \end{align*}
   \item[\rm (iii)] for $0 \leq i \leq D$,
 \begin{align*}
  Z E^*_i = E^*_i Z = E^*_i {\sf A} E^*_i  \biggl( 1 + \frac{\vartheta_i}{q+q^{-1}} \biggr);
  \end{align*}
 \item[\rm (iv)] $Z$ is central in $T$.
 \end{enumerate}
 \end{proposition}
 \begin{proof} (i)  The matrices $A, A^*, \lbrace E_i \rbrace_{i=0}^D, \lbrace E^*_i \rbrace_{i=0}^D$ are contained in $T$.\\
 \noindent (ii) By \eqref{eq:MainAssumeBF} and  $E_r E_s= \delta_{r,s} E_r$ for $0 \leq r,s\leq D$. \\
  \noindent (iii) By \eqref{eq:MainAssumeBF} and $E^*_r E^*_s= \delta_{r,s} E^*_r$ for $0 \leq r,s\leq D$. \\
 \noindent (iv) By (ii), (iii) above.
 \end{proof}
 
 \noindent We have some comments about the formula \eqref{eq:vtheta}. Recall that $\beta = q^2+q^{-2}$.
 \medskip
 
 \begin{lemma} \label{lem:thBasic}
 We have
 \begin{enumerate}
 \item[\rm (i)] $\vartheta_{i-1} - \beta \vartheta_i + \vartheta_{i+1}=0 \quad (1 \leq i \leq D-1)$;
 \item[\rm (ii)] $\vartheta^2_{i-1} - \beta \vartheta_{i-1} \vartheta_i + \vartheta^2_i = - (q^2 - q^{-2})^2 \quad (1 \leq i \leq D)$.
 \end{enumerate}
 \end{lemma}
 \begin{proof} Use \eqref{eq:vtheta}.
 \end{proof}
 
 \section{Some universal Askey-Wilson relations}
 
 \noindent We continue to discuss the implications of Assumptions \ref{def:spintype}, \ref{def:spintypex}. Recall the matrices $\sf A$, $\sf B$ from \eqref{eq:AB}.
 In this section, we show that $\sf A$ and $\sf B$ 
 satisfy some universal Askey-Wilson relations that involve the matrix $Z$. We put these universal Askey-Wilson relations
 in $\mathbb Z_3$-symmetric form. 
 \medskip
 
 \noindent We refer the reader to \cite{vidunas, Zhidd} for background information about the original Askey-Wilson relations
 and their relationship to the tridiagonal relations. The universal Askey-Wilson relations were introduced in \cite{uaw}. Information about the $\mathbb Z_3$-symmetric form can be found in
 \cite{vinet1, vinet2, vinet3, fairlie, havlicek, hwh, hwh2, hwh3, dahaZ3, nomTB, odesskii, uaw, uawsl2, uawDAHA, pseudo, TZ}.

 \begin{proposition} \label{prop:AWZ} We have
 \begin{align*}
 &{\sf A}^2 {\sf B} - \beta {\sf A}{\sf B}{\sf A}  + {\sf B} {\sf A}^2 + (q^2-q^{-2})^2 {\sf B} \\
 & \qquad = (q^2-q^{-2})^2 Z - (q-q^{-1})(q^2-q^{-2}) Z{\sf A}, \\
  &{\sf B}^2 {\sf A} - \beta {\sf B} {\sf A}{\sf B}   + {\sf A}{\sf B}^2  + (q^2-q^{-2})^2 {\sf A} \\
 & \qquad = (q^2-q^{-2})^2 Z - (q-q^{-1})(q^2-q^{-2}) Z{\sf B}. 
 \end{align*}
 \end{proposition}
 \begin{proof} To verify the first equation, let $\Delta$ denote the left-hand side minus the right-hand side. We show that $\Delta=0$.  Note that
 \begin{align*}
 \Delta = I \Delta I =   \sum_{i=0}^D \sum_{j=0}^D E_i \Delta E_j.
 \end{align*}
 For $0 \leq i, j\leq D$ we show that $E_i \Delta E_j=0$. We have
 \begin{align*}
 E_i \Delta E_j &= E_i {\sf B} E_j \Bigl( \vartheta^2_i - \beta \vartheta_i \vartheta_j + \vartheta^2_j + (q^2-q^{-2})^2 \Bigr)\\
& \qquad  - \delta_{i,j} (q^2-q^{-2})^2 \biggl( 1 - \frac{\vartheta_i}{q+q^{-1}} \biggr) E_i Z.
 \end{align*}
\noindent In this equation we examine the terms. By Lemma \ref{lem:TPR}(ii),  $E_i {\sf B} E_j=0$ if $\vert i-j\vert >1$.  By Lemma \ref{lem:thBasic}(ii), the coefficient of $E_i {\sf B} E_j$ is zero if
 $\vert i-j\vert = 1$. By these comments, $E_i \Delta E_j = 0$ if $i \not=j$. Also
\begin{align*}
E_i \Delta E_i &= E_i {\sf B} E_i 
 \Bigl( \vartheta^2_i (2- \beta) + (q^2-q^{-2})^2 \Bigr)
 - (q^2-q^{-2})^2 \biggl( 1 - \frac{\vartheta_i}{q+q^{-1}} \biggr) E_i Z
 \\
 &= (q^2 - q^{-2})^2 \biggl( 1 - \frac{\vartheta^2_i}{(q+q^{-1})^2} \biggr) E_i {\sf B} E_i
  - (q^2-q^{-2})^2 \biggl( 1 - \frac{\vartheta_i}{q+q^{-1}} \biggr) E_i Z \\
 &=(q^2-q^{-2})^2  \biggl( 1 - \frac{\vartheta_i}{q+q^{-1}} \biggr) \biggl( \biggl( 1 + \frac{\vartheta_i}{q+q^{-1}} \biggr)E_i {\sf B} E_i - E_i Z \biggr) 
 \\
 &=0,
 \end{align*}
 \noindent where the last equality holds by Proposition \ref{lem:Z1}(ii).
 We have shown that $E_i \Delta E_j=0$ for $0 \leq i,j\leq D$. Therefore $\Delta=0$, and the first equation is  verified. The second equation is similarly verified.
 \end{proof}

\begin{definition}\label{def:C} Define
\begin{align}\label{eq:C}
{\sf C} = Z - \frac{q {\sf A}{\sf  B}-q^{-1} {\sf B}{\sf A}}{q^2-q^{-2}}.
\end{align}
\end{definition}
\noindent Note that ${\sf C} \in T$.

 \begin{proposition} \label{lem:GABC} 
 We have
 \begin{align}
\label{eq:GZ2}  {\sf A} + \frac{q {\sf B}{\sf C}-q^{-1} {\sf C}{\sf B}}{q^2-q^{-2}} &= Z, \\
\label{eq:GZ3}  {\sf B}+ \frac{q {\sf C}{\sf A}-q^{-1} {\sf A}{\sf C}}{q^2-q^{-2}} &= Z, \\
\label{eq:GZ1} {\sf C} + \frac{q {\sf A}{\sf B} - q^{-1} {\sf B}{\sf A}}{q^2-q^{-2}}& = Z.
 \end{align}
 \end{proposition}
 \begin{proof} To verify \eqref{eq:GZ2} and \eqref{eq:GZ3}, eliminate ${\sf C}$ using \eqref{eq:C} and compare the results with 
 Proposition  \ref{prop:AWZ}. The relation \eqref{eq:GZ1} is a reformulation of \eqref{eq:C}.
 \end{proof}
 
\begin{remark}\label{rem:AW} \rm The universal Askey-Wilson algebra $\Delta_q$ was introduced in \cite{uaw}.
The relations \eqref{eq:GZ2}--\eqref{eq:GZ1} are a special case of the defining relations
for $\Delta_q$. This special case is discussed in \cite[Section~2.1]{vinet3} in connection with the representation theory of $U_q(\mathfrak{sl}_2)$.
See also \cite[Sections~5,14]{pseudo}.
\end{remark}
 
 \section{Some intertwiners}
 
 \noindent  We continue to discuss the implications of Assumptions \ref{def:spintype}, \ref{def:spintypex}. Recall the matrices $\sf A,B,C$ from Proposition \ref{lem:GABC}.
 In this section we have two goals. First, we find all the invertible matrices $W$ in the Bose-Mesner algebra $M$ such that 
 $W^{-1}{\sf B} W = {\sf C}$. Second, we find all the invertible matrices $W^*$ in the dual Bose-Mesner algebra $M^*$ such that  $W^* {\sf A} (W^*)^{-1} = {\sf C}$.
To reach these goals, we will adapt a strategy from \cite{pseudo}.

 
  \begin{lemma} \label{lem:GeigFORM} For $1 \leq i \leq D$ we have
 \begin{align*}
 \frac{ q \vartheta_i - q^{-1} \vartheta_{i-1}}{q^2-q^{-2}} = a q^{2i-D-1},
 \qquad \qquad
 \frac{q \vartheta_{i-1}- q^{-1} \vartheta_i}{q^2-q^{-2}} = a^{-1} q^{D-2i+1}.
 \end{align*}
 \end{lemma}
 \begin{proof} By \eqref{eq:vtheta}.
 \end{proof}

 \begin{lemma}\label{lem:GeigProd} Let $0 \leq i,j\leq D$ such that $\vert i-j \vert =1$. Then
 \begin{align*} 
 1 = \frac{q \vartheta_i - q^{-1} \vartheta_j}{q^2-q^{-2}} \,\frac{q \vartheta_j - q^{-1} \vartheta_i}{q^2-q^{-2}}.
 \end{align*}
\end{lemma}
 \begin{proof} By Lemma \ref{lem:thBasic}(ii) or Lemma \ref{lem:GeigFORM}.
 \end{proof}

  \noindent For the moment, pick any $W \in M$ and write 
 \begin{align*} 
 W = \sum_{i=0}^D \alpha_i E_i \qquad \qquad \alpha_i \in \mathbb C.
 \end{align*}
 The matrix $W$ is invertible if and only if $\alpha_i \not=0$ for $0 \leq i \leq D$. In this case,
 \begin{align*}
 W^{-1} =\sum_{i=0}^D \alpha^{-1}_i E_i.
 \end{align*}
 In the next few results, we assume that $W$ is invertible and compare $W^{-1}{\sf B}W$ with $\sf C$.

 \begin{lemma}\label{lem:Gdif1} Assume that the above matrix $W$ is invertible. Then
 \begin{align*}
 {\sf C} -W^{-1}{\sf B}W= \sum_{\stackrel{0 \leq i,j \leq D}{\vert i-j\vert=1}} E_i{\sf B}E_j \biggl( \frac{q^{-1}  \vartheta_j - q \vartheta_i}{q^2-q^{-2}}-\frac{\alpha_j}{\alpha_i}  \biggr).
 \end{align*}
 \end{lemma} 
  \begin{proof} To verify this equation, multiply the matrix ${\sf C}-W^{-1}{\sf B}W$ on the left by $I$ and on the right by $I$. Evaluate the result using $I= E_0 + E_1+ \cdots + E_D$,
 Proposition \ref{lem:Z1}(ii),  \eqref{eq:C} and the fact that $E_i {\sf B} E_j=0$ if $\vert i - j \vert >1$ $(0 \leq i,j\leq D)$.
 \end{proof}

 \begin{proposition} \label{prop:GP1} Let $W \in M$ be invertible, and write 
  \begin{align*}
 W = \sum_{i=0}^D \alpha_i E_i \qquad \qquad \alpha_i \in \mathbb C.
 \end{align*}
\noindent The following are equivalent:
\begin{enumerate}
\item[\rm (i)] $W^{-1} {\sf B} W = {\sf C}$;
\item[\rm (ii)] $\displaystyle \frac{\alpha_j}{\alpha_i} = \frac{q^{-1}\vartheta_j - q\vartheta_i}{q^2-q^{-2}}  \qquad \mbox{\rm if} \quad \vert i - j \vert =1 \qquad (0 \leq i,j\leq D)$;
\item[\rm (iii)] $ \displaystyle \frac{\alpha_{i}}{\alpha_{i-1}} = \frac{q^{-1} \vartheta_i - q \vartheta_{i-1}}{q^2-q^{-2}} \qquad \qquad (1 \leq i \leq D)$;
\item[\rm (iv)] $ \displaystyle \frac{\alpha_{i}}{\alpha_{i-1}} = -a^{-1} q^{D-2i+1} \qquad \qquad (1 \leq i \leq D)$;
\item[\rm (v)] $ \alpha_i = \alpha_0 (-1)^i a^{-i} q^{i(D-i)}  \qquad \qquad (0 \leq i \leq D)$.
\end{enumerate}
\end{proposition}
\begin{proof} ${\rm (i)} \Leftrightarrow {\rm (ii)}$ By Lemma \ref{lem:linDep}(ii) and Lemma  \ref{lem:Gdif1}. \\
\noindent  ${\rm (ii)} \Leftrightarrow {\rm (iii)}$ By Lemma \ref{lem:GeigProd}. \\
\noindent  ${\rm (iii)} \Leftrightarrow {\rm (iv)}$ By Lemma  \ref{lem:GeigFORM}. \\
\noindent  ${\rm (iv)} \Leftrightarrow {\rm (v)}$ This is routinely checked.
\end{proof}
 
  \noindent  
 We turn our attention to $M^*$. Our next goal is to find  all
   the invertible matrices $W^* \in M^*$ such that $W^* {\sf A} (W^*)^{-1} = {\sf C}$.
\medskip

  \noindent For the moment, pick any $W^* \in M^*$ and write 
 \begin{align*}
 W^* = \sum_{i=0}^D \alpha^*_i E^*_i \qquad \qquad \alpha^*_i \in \mathbb C.
 \end{align*}
 The matrix $W^*$ is invertible if and only if $\alpha^*_i \not=0$ for $0 \leq i \leq D$. In this case,
 \begin{align*} 
 (W^*)^{-1} =\sum_{i=0}^D (\alpha^*_i)^{-1} E^*_i.
 \end{align*}
\noindent In the next few results, we assume that  $W^*$ is invertible and compare $W^* {\sf A} (W^*)^{-1}$ with $\sf C$.

 \begin{lemma} \label{lem:Gdif2} Assume that the above matrix $W^*$ is invertible. Then
 \begin{align*}
 {\sf C} -W^* {\sf A} (W^*)^{-1}=\sum_{\stackrel{0 \leq i,j \leq D}{\vert i-j\vert=1}} E^*_i {\sf A}E^*_j \biggl(  \frac{q^{-1} \vartheta_i - q \vartheta_j}{q^2-q^{-2}} - \frac{\alpha^*_i}{\alpha^*_j}   \biggr).
 \end{align*}
 \end{lemma} 
  \begin{proof} Similar to the proof of Lemma \ref{lem:Gdif1}.
 \end{proof}

 \begin{proposition} \label{prop:GP2} Let $W^* \in M^*$ be invertible, and write 
  \begin{align*}
 W^* = \sum_{i=0}^D \alpha^*_i E^*_i \qquad \qquad \alpha^*_i \in \mathbb C.
 \end{align*}
\noindent The following are equivalent:
\begin{enumerate}
\item[\rm (i)] $W^* {\sf A} (W^*)^{-1} = {\sf C}$;
\item[\rm (ii)] $\displaystyle \frac{\alpha^*_i}{\alpha^*_j} = \frac{q^{-1}\vartheta_i - q \vartheta_j}{q^2-q^{-2}}  \qquad \mbox{\rm if} \quad \vert i - j \vert =1 \qquad (0 \leq i,j\leq D)$;
\item[\rm (iii)] $ \displaystyle \frac{\alpha^*_{i}}{\alpha^*_{i-1}} = \frac{q^{-1} \vartheta_i - q \vartheta_{i-1}}{q^2-q^{-2}} \qquad \qquad (1 \leq i \leq D)$;
\item[\rm (iv)] $ \displaystyle \frac{\alpha^*_{i}}{\alpha^*_{i-1}} = -a^{-1} q^{D-2i+1} \qquad \qquad (1 \leq i \leq D)$;
\item[\rm (v)] $ \alpha^*_i = \alpha^*_0 (-1)^i a^{-i} q^{i(D-i)}  \qquad \qquad (0 \leq i \leq D)$.
\end{enumerate}
\end{proposition}
\begin{proof} Similar to the proof of Proposition \ref{prop:GP1}.
\end{proof}

 \section{The Boltzmann pair $\sf W, W^*$}
 
 \noindent We continue to discuss the implications of Assumptions \ref{def:spintype}, \ref{def:spintypex}.
 Recall the matrices $\sf A,B,C$ from Proposition \ref{lem:GABC}. In this section, we introduce a pair of matrices
 ${\sf W} \in M$ and ${\sf W^*} \in M^*$ such that
 $\sf W^{-1} B W = C=W^* A (W^*)^{-1} $. We obtain some formulas involving $\sf W$ and $\sf W^*$.
\medskip

 \noindent  The following definition is motivated by Propositions \ref{prop:GP1}(v), \ref{prop:GP2}(v).
 
  \begin{definition} \label{def:Gtau} \rm Define the complex scalars  
   \begin{align*}
 \tau_i = (-1)^i a^{-i}q^{i(D-i)} \qquad \qquad (0 \leq i \leq D).
 \end{align*}
 \end{definition}
 \noindent With reference to Definition \ref{def:Gtau}, we have $\tau_i \not=0$ for $0 \leq i \leq D$. Moreover
 \begin{align*}
 \tau_0=1, \qquad \qquad \frac{\tau_i}{\tau_{i-1}} = -a^{-1} q^{D-2i+1} \qquad \qquad (1 \leq i \leq D).
 \end{align*}

 
 \begin{definition} \label{def:GWWs} \rm Pick $0 \not=f \in \mathbb C$. Define the matrices
  \begin{align}
 {\sf W} = f \sum_{i=0}^D \tau_i E_i, \qquad \qquad {\sf W^*} = f \sum_{i=0}^d \tau_i E^*_i. \label{eq:GWWS}
 \end{align}
Following \cite[Definition~1.2]{C:spinLP}  we call the pair $\sf W, W^*$ a {\it Boltzmann pair}.
 \end{definition}
 
 \noindent We have some comments about Definition \ref{def:GWWs}. Observe that ${\sf W} \in M$ and ${\sf W^*} \in M^*$. The matrix $\sf W$ (resp. $\sf W^*$) is symmetric (resp. diagonal).
 The matrices $\sf W, W^*$ are invertible, and 
  \begin{align}  \label{eq:GWWsInv}
 {\sf W}^{-1} = f^{-1} \sum_{i=0}^D \tau^{-1}_i E_i, \qquad \qquad ({\sf W}^*)^{-1} = f^{-1} \sum_{i=0}^D \tau^{-1}_i E^*_i.
 \end{align}
 We have
 \begin{align} \label{eq:GWWcom}
 \sf W A = A W, \qquad \qquad W^* B = B W^*.
 \end{align}
 For $0 \leq i \leq D$ we have
 \begin{align}
& {\sf W}E_i = E_i{\sf W} = f \tau_i E_i, \qquad \qquad {\sf W}^{-1} E_i = E_i {\sf W}^{-1} = f^{-1} \tau^{-1}_i E_i, \label{eq:GWE} \\
& {\sf W}^*E^*_i = E^*_i{\sf W}^* = f \tau_i E^*_i, \qquad  ({\sf W}^*)^{-1} E^*_i = E^*_i ({\sf W}^*)^{-1} = f^{-1} \tau^{-1}_i E^*_i.\label{eq:GWsE} 
 \end{align}
By Propositions \ref{prop:GP1}, \ref{prop:GP2},
 \begin{align}
 \sf W ^{-1} B  W = C, \qquad \qquad  W^* A( W^*)^{-1} =  C.  \label{eq:GWWs}
 \end{align}
 
 \noindent Next, we give some formulas involving $\sf W$ and $\sf W^*$. We will use these formulas shortly.
\medskip

\noindent Recall the valencies $\lbrace k_i \rbrace_{i=0}^D$ from  \eqref{eq:kicibi}.
  \begin{proposition} \label{lem:GEWE} We have
   \begin{align}\label{eq:GWk}
  &E^*_0 {\sf W} E^*_0  = \frac{ f \sum_{i=0}^D \tau_i k_i}{\vert X \vert} E^*_0,
   \qquad \qquad 
   E^*_0 {\sf W}^{-1} E^*_0  = \frac{  \sum_{i=0}^D \tau^{-1}_i k_i}{f\vert X \vert} E^*_0,
   \\
   &E_0 {\sf W}^* E_0  = \frac{ f \sum_{i=0}^D \tau_i k_i}{\vert X \vert} E_0,
   \qquad \qquad 
  E_0 ({\sf W}^*)^{-1} E_0  = \frac{ \sum_{i=0}^D \tau^{-1}_i k_i}{f\vert X \vert} E_0.
  \label{eq:GWk2}
  \end{align}
  \end{proposition} 
  \begin{proof}  To verify the equation on the left in \eqref{eq:GWk}, eliminate $\sf W$ using
  Definition  \ref{def:GWWs} and evaluate the result using Lemma
  \ref{lem:r5} along with $k_i = k^*_i$ $(0 \leq i \leq D)$. 
 To verify the equation on the right in \eqref{eq:GWk}, 
 eliminate $\sf W^{-1}$ using
  \eqref{eq:GWWsInv} and evaluate the result using Lemma
  \ref{lem:r5} along with $k_i = k^*_i$ $(0 \leq i \leq D)$. 
 The equations \eqref{eq:GWk2} are similarly verified.
\end{proof} 

\section{The isomorphism $\rho: T\to T$}

 \noindent We continue to discuss the implications of Assumptions \ref{def:spintype}, \ref{def:spintypex}.
 Recall the matrices $\sf A,B,C$ from Proposition \ref{lem:GABC}. These matrices generate the subconstituent algebra $T$.
 Recall the Boltzmann pair $\sf W, W^*$ 
 from Definition   \ref{def:GWWs}. In this section, we use $\sf W, W^*$ to
  construct an algebra isomorphism $\rho: T\to T$ that sends $\sf A\mapsto B\mapsto C\mapsto A$. 
  
 \begin{lemma} \label{lem:GrhoIso} There exists an algebra isomorphism
 $\rho:T\to T$ that sends
 \begin{align*}
 S \mapsto ({\sf W^*W})^{-1} S ({\sf W^*W}) 
 \end{align*}
 \noindent for all $S \in T$.
 \end{lemma}
 \begin{proof} The matrices $\sf W,  W^*$ are invertible and contained in $T$.
 \end{proof}

 \begin{proposition} \label{lem:GrhoMap} The isomorphism $\rho$ sends
 \begin{align*}
\sf A \mapsto B \mapsto C \mapsto A.
 \end{align*}
 \end{proposition}
 \begin{proof} First, we show that $\rho$ sends $\sf A \mapsto B$. It suffices to show 
 \begin{align}
 \sf A W^* W = W^* W B. \label{eq:Gstep1}
 \end{align}
By  \eqref{eq:GWWs},
 \begin{align*}
\sf W^{-1} B W = W^* A (W^*)^{-1}.
 \end{align*}
 Rearranging the terms, we get
 \begin{align*}
\sf  W W^* A = B W W^*.
 \end{align*} In this equation we take the transpose of each side, and obtain \eqref{eq:Gstep1}. 
 \medskip
 
 \noindent Next, we show that $\rho$ sends $\sf B \mapsto C$. It suffices to show
 \begin{align*} 
\sf  B W^* W = W^* W C.
 \end{align*}
 Using  \eqref{eq:GWWcom}, \eqref{eq:GWWs} we obtain
 \begin{align*}
\sf B W^* W = W^* B W  = W^* W W^{-1} B W = W^* W C.
 \end{align*}
 Next, we show that $\rho$ sends $\sf C \mapsto A$. 
 It suffices to show
 \begin{align*}
 \sf C W^* W = W^* W A.
 \end{align*}
  Using  \eqref{eq:GWWcom}, \eqref{eq:GWWs} we obtain
 \begin{align*}
 \sf CW^* W = W^* A (W^*)^{-1} W^* W = W^* A W = W^* W A.
 \end{align*}
  \end{proof}
  
  \section{A braid relation for $\sf W$ and $\sf W^*$}

 \noindent We continue to discuss the implications of Assumptions \ref{def:spintype}, \ref{def:spintypex}.
 Recall the Boltzmann pair $\sf W, W^*$ 
 from Definition   \ref{def:GWWs}. In this section, we obtain the braid relation
  \begin{align*}
 \sf W W^* W  = W^* W W^*.
  \end{align*}
  We acknowledge that the above relation appeared earlier in the context of spin models; see \cite{mun} and also
  \cite[Corollary~5.4]{CW}, \cite[Lemma~11.6]{nomSpinModel}.
  \medskip
  
  \noindent Recall the isomorphism $\rho$ from Lemma \ref{lem:GrhoIso}.
  
  \begin{lemma} \label{lem:GrhoE} The following hold:
  \begin{enumerate}
  \item[\rm (i)] $\rho(E_i) = E^*_i$ $(0 \leq i \leq D)$;
  \item[\rm (ii)] $\rho({\sf W}) = {\sf W}^*$.
  \end{enumerate}
  \end{lemma}
  \begin{proof} (i) By \eqref{eq:AsumBsum} and
 linear algebra,
       \begin{align*}
  E_i=\prod_{\stackrel{0 \leq j \leq D}{j \neq i}}
         \frac{{\sf A}-\vartheta_jI}{\vartheta_i-\vartheta_j},
         \qquad \qquad
           E^*_i=\prod_{\stackrel{0 \leq j \leq D}{j \neq i}}
         \frac{{\sf B}-\vartheta_jI}{\vartheta_i-\vartheta_j}.
         \end{align*}
We have $\rho({\sf A})={\sf B}$. By these comments $\rho(E_i)=E^*_i$.
\\
\noindent (ii)  By Definition \ref{def:GWWs} and (i) above,
\begin{align*}
\rho({\sf W}) = f \sum_{i=0}^D \tau_i \rho(E_i) = f \sum_{i=0}^D \tau_i E^*_i = {\sf W}^*.
\end{align*}
  \end{proof}
  
  \begin{proposition}\label{prop:GWWW} We have
 \begin{align}
 \label{eq:GWWW}
\sf  W W^* W = W^* W W^*.
\end{align}
\end{proposition}
\begin{proof} This is Lemma \ref{lem:GrhoE}(ii) in disguise. Evaluating Lemma  \ref{lem:GrhoE}(ii) using    Lemma \ref{lem:GrhoIso}, we obtain
\begin{align*}
\sf (W^*W)^{-1} W (W^* W) = W^*.
\end{align*}
Rearrange the terms in this equation to get \eqref{eq:GWWW}.
\end{proof}

\noindent We mention some variations on Lemma  \ref{lem:GrhoE}(i).
  \begin{lemma} \label{lem:GWWcon}
  The following hold for $0 \leq i \leq D$:
  \begin{enumerate}
   \item[\rm (i)] ${\sf W} E^*_i {\sf W}^{-1} = ({\sf W}^*)^{-1} E_i {\sf W}^*$;
   \item[\rm (ii)] ${\sf W}^{-1} E^*_i {\sf W} = {\sf W}^* E_i ({\sf W}^*)^{-1}$.
  \end{enumerate}
  \end{lemma}
  \begin{proof} (i) This is Lemma \ref{lem:GrhoE}(i) in disguise. \\
\noindent (ii) For the equation in part (i), take the transpose of each side.
\end{proof}

\section{The matrix $\sf W$ is type II}

 \noindent We continue to discuss the implications of Assumptions \ref{def:spintype}, \ref{def:spintypex}.
 Recall that the distance matrices $\lbrace A_i \rbrace_{i=0}^D$ form a basis for the Bose-Mesner algebra $M$, and the dual distance matrices $\lbrace  A^*_i \rbrace_{i=0}^D$ form
 a basis for the dual Bose-Mesner algebra $M^*$.
 Recall the Boltzmann pair $\sf W, W^*$ 
 from Definition   \ref{def:GWWs}. In this section, we express $\sf W^{\pm 1} $ as a linear combination of $\lbrace A_i \rbrace_{i=0}^D$,
 and $(\sf W^*)^{\pm 1}$ as a linear combination of  $\lbrace A^*_i \rbrace_{i=0}^D$. We use the results to show that $\sf W$ is type II.

   \begin{lemma} \label{lem:Xprod} We have
   \begin{align}\label{eq:GNZ}
            \vert X \vert = \Biggl( \sum_{i=0}^D \tau_i k_i \Biggr)\Biggl( \sum_{i=0}^D \tau^{-1}_i k_i \Biggr).
            \end{align}
             \end{lemma}
  \begin{proof}  Setting $i=0$ in Lemma  \ref{lem:GWWcon}(i), we obtain
  \begin{align*}
  {\sf W} E^*_0 {\sf W}^{-1} = ({\sf W}^*)^{-1} E_0 {\sf W}^*.
  \end{align*}
  In this equation, we multiply each side on the left by $E^*_0$ and on the right by $E^*_0$; this yields
  \begin{align} \label{eq:Gst}
   E^*_0 {\sf W} E^*_0 {\sf W}^{-1} E^*_0 &= E^*_0 ({\sf W}^*)^{-1} E_0 {\sf W}^* E^*_0.
   \end{align}
   \noindent We examine each side of \eqref{eq:Gst}. 
   Using \eqref{eq:GWk} and $(E^*_0)^2=E^*_0$,
   \begin{align*}
    E^*_0 {\sf W} E^*_0 {\sf W}^{-1} E^*_0 &= \bigl(E^*_0 {\sf W} E^*_0\bigr) \bigl( E^*_0 {\sf W}^{-1} E^*_0\bigr) \\
   &= \frac{f \sum_{i=0}^D \tau_i k_i}{\vert X \vert} E^*_0 \frac{ \sum_{i=0}^D \tau^{-1}_i k_i}{f\vert X \vert} E^*_0 \\
   &= \vert X \vert^{-2} \Biggl( \sum_{i=0}^D \tau_i k_i \Biggr) \Biggl( \sum_{i=0}^D \tau^{-1}_i k_i \Biggr) E^*_0.
   \end{align*}
   Using Lemma \ref{lem:r6} and \eqref{eq:GWsE},
   \begin{align*}
   E^*_0 ({\sf W}^*)^{-1} E_0 {\sf W}^* E^*_0
   = (f^{-1} E^*_0) E_0 (f E^*_0) 
    = E^*_0 E_0 E^*_0 
    = \vert X \vert^{-1} E^*_0.
   \end{align*}
   \noindent Evaluating \eqref{eq:Gst} using the above comments, we obtain
   \begin{align*}
   \vert X \vert^{-2} \Biggl( \sum_{i=0}^D \tau_i k_i \Biggr) \Biggl( \sum_{i=0}^D \tau^{-1}_i k_i \Biggr) E^*_0 = \vert X \vert^{-1} E^*_0.
   \end{align*}
   This implies \eqref{eq:GNZ}. 
     \end{proof}
     
     \begin{corollary} We have
        \begin{align*} 
        \sum_{i=0}^D \tau_i k_i \not=0, \qquad \qquad \sum_{i=0}^D \tau^{-1}_i k_i \not=0.
   \end{align*}
\end{corollary}
     \begin{proof} By  \eqref{eq:GNZ}.
     \end{proof}

  \begin{proposition}\label{prop:GAi}
  We have
  \begin{align} \label{eq:Gfrac12}
  &{\sf W} = f \frac{ \sum_{i=0}^D \tau^{-1}_i A_i}{\sum_{i=0}^D \tau^{-1}_i k_i},
  \qquad \qquad
 {\sf W}^{-1} = \frac{1}{f}\, \frac{ \sum_{i=0}^D \tau_i A_i}{\sum_{i=0}^D \tau_i k_i}, \\
  &
 {\sf  W}^* = f \frac{ \sum_{i=0}^D \tau^{-1}_i A^*_i}{\sum_{i=0}^D \tau^{-1}_i k_i}, 
  \qquad \qquad 
  ({\sf W}^*)^{-1} = \frac{1}{f}\, \frac{ \sum_{i=0}^D \tau_i A^*_i}{\sum_{i=0}^D \tau_i k_i}.
  \label{eq:Gfrac34}
  \end{align}
  \end{proposition}
  \begin{proof} First, we verify the equation on the left in \eqref{eq:Gfrac12}. Define
   \begin{align*}
  S= \Biggl(\sum_{i=0}^D \tau^{-1}_i k_i \Biggr) {\sf W} - f\sum_{i=0}^D \tau^{-1}_i A_i.
  \end{align*}
We show that $S=0$.
  Setting $i=0$ in Lemma  \ref{lem:GWWcon}(ii), we obtain
  \begin{align*}
 {\sf W}^{-1} E^*_0 {\sf W} = {\sf W}^* E_0 ({\sf W}^*)^{-1}.
  \end{align*}
  Therefore
    \begin{align} \label{eq:Gtwosides}
  E^*_0 {\sf W}^{-1} E^*_0 {\sf W} = E^*_0 {\sf W}^* E_0 ({\sf W}^*)^{-1}.
  \end{align}
  We examine each side of \eqref{eq:Gtwosides}. By \eqref{eq:GWk},
  \begin{align*}
   E^*_0 {\sf W}^{-1} E^*_0 {\sf W} = \frac{\sum_{i=0}^D \tau^{-1}_i k_i}{f \vert X \vert} E^*_0{\sf W}.
  \end{align*}
  By Lemma \ref{lem:red1} and     \eqref{eq:GWWsInv},   \eqref{eq:GWsE},
  \begin{align*}
  E^*_0 {\sf W}^* E_0 ({\sf W}^*)^{-1} &= f E^*_0 E_0 ({\sf W}^*)^{-1} 
  = f E^*_0 E_0 \Biggl( f^{-1} \sum_{i=0}^D \tau^{-1}_i E^*_i\Biggr) \\
  & =\sum_{i=0}^D \tau^{-1}_i E^*_0 E_0 E^*_i 
  = \vert X \vert^{-1} \sum_{i=0}^D \tau^{-1}_i E^*_0 A_i.
  \end{align*}
  Evaluating \eqref{eq:Gtwosides} using the above comments, we obtain $0=E^*_0S$. By construction $S \in M$, so $S=0$ in view of Lemma \ref{lem:need}.
  We have verified the equation on the left in \eqref{eq:Gfrac12}. \\
  \noindent Next, we verify the equation on the right in \eqref{eq:Gfrac12}.
  Define
   \begin{align*}
  S'=f\Biggl(\sum_{i=0}^D \tau_i k_i \Biggr) {\sf W}^{-1} - \sum_{i=0}^D \tau_i A_i.
  \end{align*}
We show that $S'=0$.
  Setting $i=0$ in Lemma  \ref{lem:GWWcon}(i), we obtain
  \begin{align*}
 {\sf  W} E^*_0 {\sf W}^{-1} = ({\sf W}^*)^{-1} E_0 {\sf W}^*.
  \end{align*}
  Therefore
    \begin{align} \label{eq:Gtwosides2}
  E^*_0 {\sf W} E^*_0 {\sf W}^{-1} = E^*_0 ({\sf W}^*)^{-1} E_0{\sf  W}^*.
  \end{align}
  We examine each side of \eqref{eq:Gtwosides2}.
  By \eqref{eq:GWk},
  \begin{align*}
   E^*_0 {\sf W} E^*_0 {\sf W}^{-1} = \frac{f\sum_{i=0}^D \tau_i k_i}{ \vert X \vert} E^*_0 {\sf W}^{-1}.
  \end{align*}
   By Lemma \ref{lem:red1} and   \eqref{eq:GWWS},    \eqref{eq:GWsE},
    \begin{align*}
  E^*_0 ({\sf W}^*)^{-1} E_0 {\sf W}^* &= f^{-1} E^*_0 E_0 {\sf W}^* 
  = f^{-1}E^*_0 E_0 \Biggl( f \sum_{i=0}^D \tau_i E^*_i\Biggr) \\
  & =\sum_{i=0}^D \tau_i E^*_0 E_0 E^*_i 
  = \vert X \vert^{-1} \sum_{i=0}^D \tau_i E^*_0 A_i.
  \end{align*}
\noindent  Evaluating \eqref{eq:Gtwosides2} using the above comments, we obtain $0=E^*_0S'$. By construction $S' \in M$, so $S'=0$ in view of Lemma \ref{lem:need}.
  We have verified the equation on the right in \eqref{eq:Gfrac12}.
  \\
\noindent The equations in \eqref{eq:Gfrac34} are similarly verified.
  \end{proof}
  
  \noindent Recall the all 1's matrix $J \in {\rm Mat}_X(\mathbb C)$.
  
  \begin{proposition}\label{cor:GentryW} We have
  \begin{align*}
         {\sf W} \circ {\sf W}^{-1} = \vert X \vert^{-1} J.
  \end{align*}
  \end{proposition}
  \begin{proof} To verify this equation, eliminate $\sf W$ and $\sf W^{-1}$ using  \eqref{eq:Gfrac12}, and evaluate
  the result using Lemma \ref{lem:Xprod} and $J=\sum_{i=0}^D A_i$.
   \end{proof}
  
  \begin{proposition} \label{prop:GTYPE2}  The matrix $\sf W$  is type II.
  \end{proposition}
  \begin{proof} By Lemma \ref{lem:type2Char}, Proposition \ref{cor:GentryW}, and since $\sf W$ is symmetric.
  \end{proof}

  \noindent We mention a result for later use. Recall that $\sf W^*$ depends on the base vertex $x \in X$.
  \begin{lemma}  Pick $y \in X$ and write $\ell = \partial(x,y)$. Then
 \begin{align}
   {\sf  W}_{x,y} = \frac{f \tau^{-1}_\ell}{\sum_{i=0}^D \tau^{-1}_i k_i}, \qquad \qquad 
    {\sf W}^*_{y,y} = f \tau_\ell. \label{eq:GWxy}
 \end{align}    
 \noindent Moreover,
  \begin{align} \label{eq:GWPW}
 {\sf W}_{x,y}  {\sf W}^*_{y,y} =  \frac{f^2}{\sum_{i=0}^D \tau^{-1}_i k_i}.
  \end{align}
  \end{lemma}
 \begin{proof}  The equations in \eqref{eq:GWxy} come from
 \begin{align*}
 {\sf   W} = f \frac{\sum_{i=0}^D \tau^{-1}_i A_i}{\sum_{i=0}^D \tau^{-1}_i k_i}, \qquad \qquad
 {\sf W}^* = f \sum_{i=0}^D \tau_i E^*_i.
  \end{align*}
 Equation \eqref{eq:GWPW} follows from \eqref{eq:GWxy}.
 \end{proof}
  
  \section{Strengthening Assumption \ref{def:spintypex} }
  
 \noindent  In this section, Assumption  \ref{def:spintype} remains in effect but we will strengthen Assumption  \ref{def:spintypex}.
 Under the strengthened assumption, 
 we will show that the matrix $\sf W$ from Definition \ref{def:GWWs}  is a spin model for an appropriate choice of $f$.
  \medskip
  
  \noindent Throughout this section, the following assumption is in effect in addition to Assumption \ref{def:spintype}.
  
   \begin{assumption} \label{def:spintype2} \rm We refer to Assumption \ref{def:spintype}.
   Let $A$ denote the adjacency matrix of $\Gamma$.
   We assume that for all $x \in X$,
    \begin{align}
 \label{eq:MainAssume2}
 \begin{split}
& \sum_{i=0}^D E^*_i \frac{A -\varepsilon I}{\alpha} E^*_i  \biggl(1+ \frac{\theta_{i}-\varepsilon}{\alpha}\,\frac{1}{q+q^{-1}}\biggr) \\
& =
  \sum_{i=0}^D E_i \frac{A^* -\varepsilon I}{\alpha} E_i  \biggl(1+ \frac{\theta_{i}-\varepsilon}{\alpha}\,\frac{1}{q+q^{-1}}\biggr),
  \end{split}
\end{align}
 where $A^*=A^*(x)$ and $E^*_i= E^*_i(x)$ for $0 \leq i\leq D$.
 \end{assumption}

  \noindent  Recall the matrix $\sf W$ from Definition \ref{def:GWWs}.
  We have seen that
  \begin{align*}
 & {\sf W} = f \sum_{i=0}^D \tau_i E_i, \qquad \qquad {\sf W}^{-1} = f^{-1} \sum_{i=0}^D \tau^{-1}_i E_i,
  \\
  &{\sf W} = f \frac{ \sum_{i=0}^D \tau^{-1}_i A_i}{\sum_{i=0}^D \tau^{-1}_i k_i},
  \qquad \qquad
  {\sf W}^{-1} = \frac{1}{f}\, \frac{ \sum_{i=0}^D \tau_i A_i}{\sum_{i=0}^D \tau_i k_i}.
  \end{align*}
  The parameter $0 \not= f \in \mathbb C$ is free.
We will show that $\sf W$ is a spin model, provided that
 \begin{align*}
  f^2 = \vert X \vert^{1/2} \sum_{i=0}^D \tau^{-1}_i k_i.
  \end{align*}

 \begin{proposition} \label{lem:GWWWE} For  $a,b,c \in X$ we have
  \begin{align} \label{eq:GST}
  \sum_{e \in X} \frac{{\sf W}_{e,b} {\sf W}_{e,c}}{{\sf W}_{e,a}} =  \frac{f^2}{\sum_{i=0}^D \tau^{-1}_i k_i} \,\frac{{\sf W}_{b,c}}{{\sf W}_{a,b}{\sf W}_{c,a}}.
  \end{align}    
  \end{proposition}
       \begin{proof}  For notational convenience, define
       \begin{align*}
       s =  \frac{f^2}{\sum_{i=0}^D \tau^{-1}_i k_i}.
       \end{align*}
       Consider the matrix ${\sf W}^*={\sf W}^*(x)$ from Definition \ref{def:GWWs}, where we choose $x=a$.
       We compute the $(b,c)$-entry for either side of
  \begin{align*}
\sf  W W^* W = W^* W W^*.
  \end{align*}
  Using matrix multiplication and \eqref{eq:GWPW},
  \begin{align*}
 \bigl( {\sf W W^* W} \bigr)_{b,c} &= \sum_{e \in X} {\sf W}_{b,e} {\sf W}^*_{e,e} {\sf W}_{e,c} \\
                                             & = \sum_{e \in X} {\sf W}_{b,e}\, \frac{s}{{\sf W}_{a,e}}\,{\sf W}_{e,c} \\
                                              & = s \sum_{e \in X} \frac{ {\sf W}_{e,b} {\sf W}_{e,c}}{ {\sf W}_{e,a}}.
\end{align*}
  \noindent Using matrix multiplication and \eqref{eq:GWPW},
  \begin{align*}
  \bigl( {\sf W^* W W^*}\bigr)_{b,c} &= {\sf W}^*_{b,b} {\sf W}_{b,c} {\sf W}^*_{c,c} \\
                         &=\frac{s}{{\sf W}_{a,b}}\, {\sf W}_{b,c}\, \frac{s}{{\sf W}_{a,c}} \\
                             &=s^2 \,\frac{{\sf W}_{b,c}}{{\sf W}_{a,b}{\sf W}_{c,a}}.
  \end{align*}
  By these comments we obtain \eqref{eq:GST}.
\end{proof}

  \begin{theorem} \label{thm:main} For the matrix $\sf W$ in Definition \ref{def:GWWs}, the following are equivalent:
  \begin{enumerate}
  \item[\rm (i)] $\sf W$ is a spin model;
\item[\rm (ii)] $  f^2 = \vert X \vert^{1/2} \sum_{i=0}^D \tau^{-1}_i k_i$.
\end{enumerate}
  \end{theorem}
  \begin{proof} Consider the conditions (i)--(iii) in Definition \ref{def:spinM}.  Definition \ref{def:spinM}(i) holds by construction. Definition \ref{def:spinM}(ii) holds by Proposition \ref{prop:GTYPE2}.
  Definition \ref{def:spinM}(iii) holds if and only if  $ f^2 = \vert X \vert^{1/2} \sum_{i=0}^D \tau^{-1}_i k_i$,
  in view of  Proposition  \ref{lem:GWWWE}.
  The result follows.
   \end{proof}
   
   \noindent  Next we check that the spin model $\sf W$ in Theorem \ref{thm:main} is afforded by $\Gamma$.
   
   \begin{lemma} \label{lem:GsupportG} Referring to Theorem \ref{thm:main},  assume that the equivalent conditions {\rm (i), (ii)} hold.
   Then ${\sf W} \in M \subseteq N({\sf W})$.
   \end{lemma} 
   \begin{proof} By construction ${\sf W}\in M$.  It was shown in  \cite[Theorem~13.9]{nomSpinModel} that $M \subseteq N({\sf W})$ is a consequence of Lemma  \ref{lem:GWWcon}(ii) holding
   for every choice of  base vertex $x \in X$.
   \end{proof}

\noindent We have shown that the spin model $\sf W$ in Theorem \ref{thm:main} is afforded by $\Gamma$. Some algebraic consequences of
$\sf W$ are listed in the Appendix.

  \section{Combinatorial aspects of the central element $Z$}
  
Throughout this section, Assumptions \ref{def:spintype}, \ref{def:spintype2} are in effect. We fix a vertex $x \in X$ and write $T=T(x)$. Recall the matrix $Z=Z(x)$ from Definition \ref{def:Z}. Our goal for this section
is to describe the combinatorial implications of $Z$. In our discussion, we will often use the data in the Appendix.

\begin{lemma} \label{lem:c1} For $1 \leq i \leq D$,
\begin{align*}
E^*_{i-1} {\sf  A} E^*_{i-1} {\sf A} E^*_i \biggl(1+ \frac{\vartheta_{i-1}}{q+q^{-1}}\biggr) = E^*_{i-1} {\sf  A} E^*_i {\sf A} E^*_i \biggl(1+ \frac{\vartheta_i}{q+q^{-1}}\biggr).
\end{align*}
\end{lemma}
\begin{proof} The matrix $Z$ is central in $T$, so 
\begin{align*}
     Z E^*_{i-1} {\sf A} E^*_i =  E^*_{i-1} {\sf A} E^*_i Z.
\end{align*}
The result follows in view of Proposition \ref{lem:Z1}(iii).
\end{proof}

\begin{lemma} \label{lem:c2} For $1 \leq i \leq D$,
\begin{align*}
E^*_{i-1} A E^*_{i-1} A E^*_i \biggl(1+ \frac{\vartheta_{i-1}}{q+q^{-1}}\biggr) = E^*_{i-1} A E^*_i A E^*_i \biggl(1+ \frac{\vartheta_i}{q+q^{-1}}\biggr) + E^*_{i-1} A E^*_i \frac{\vartheta_{i-1} - \vartheta_i}{q+q^{-1}} \varepsilon.
\end{align*}
\end{lemma}
\begin{proof} Eliminate $\sf A$ in Lemma \ref{lem:c1} using \eqref{eq:AB}, and simplify the result.
\end{proof}

\noindent For $y,z \in X$ we will examine the $(y,z)$-entry on either side of the equation in Lemma \ref{lem:c2}.  We will assume that $y \in \Gamma_{i-1}(x)$ and $z \in \Gamma_i(x)$ and $\partial(y,z) \in \lbrace 1,2\rbrace$;
otherwise the $(y,z)$-entry of each term is zero.
\medskip

\noindent For $y \in X$ we abbreviate $\Gamma(y) = \Gamma_1(y)$.

\begin{lemma} \label{lem:c5} Pick an integer $i$ $(1 \leq i \leq D)$, and pick $y,z \in X$ such that
\begin{align*}
\partial(x,y) = i-1, \qquad \quad \partial(x,z) = i, \qquad \quad \partial (y,z)=1.
\end{align*}
Define
\begin{align*}
z_i = \bigl \vert \Gamma_{i-1} (x) \cap \Gamma(y) \cap \Gamma(z) \bigr \vert.
\end{align*}
Then
 \begin{align}
z_i \biggl(1+ \frac{\vartheta_{i-1}}{q+q^{-1}}\biggr) = \bigl(a_1-z_i \bigr) \biggl(1+ \frac{\vartheta_i}{q+q^{-1}}\biggr) + \frac{\vartheta_{i-1} - \vartheta_i}{q+q^{-1}} \varepsilon.    \label{eq:zif}
\end{align}
\end{lemma}
  \begin{proof} For the equation in Lemma \ref{lem:c2}, compute the $(y,z)$-entry of each side.
  \end{proof}
\noindent  In a moment, we will solve \eqref{eq:zif} for $z_i$. The following result will be useful.
 \begin{lemma} \label{lem:dnz} For $1 \leq i \leq D$ we have
 \begin{align*}
 2(q+q^{-1})+\vartheta_{i-1}+\vartheta_i = a^{-1}q^{2i-D-1}(q+q^{-1})(a+q^{D-2i+1})^2.
 \end{align*}
 Moreover 
 \begin{align*}
 2(q+q^{-1})+\vartheta_{i-1}+\vartheta_i \not=0.
 \end{align*}
 \end{lemma}
 \begin{proof} By \eqref{eq:vtheta} and    \eqref{eq:ineq1}, \eqref{eq:ineq2}.
 \end{proof}
 \noindent The following result is essentially due to Curtin and Nomura \cite[Lemma~3.5]{CNhom}.

   \begin{proposition} \label{lem:c6} {\rm (See \cite[Lemma~3.5]{CNhom}).} With reference to Lemma \ref{lem:c5}, 
  \begin{align}
  z_i =\frac{ a_1 (q+q^{-1} + \vartheta_i) + \varepsilon (\vartheta_{i-1}-\vartheta_i)}{2(q+q^{-1})+ \vartheta_{i-1}+\vartheta_i} \qquad \qquad (1 \leq i \leq D). \label{eq:ziraw}
 \end{align}
   \end{proposition}
   \begin{proof}
  Solve  \eqref{eq:zif} for $z_i$ using Lemma \ref{lem:dnz}.
  \end{proof}
  \begin{remark}\rm
   By Proposition \ref{lem:c6}, the integer $z_i$ in Lemma \ref{lem:c5} is independent of the choice of $x,y,z$.
  \end{remark}
  
  \begin{lemma} \label{lem:c7} For $1\leq i \leq D$ we have
  \begin{align*}
  z_i &= \frac{a_1}{1-aq^{1-D}} \, \frac{1-q^{2-2i}}{1+ a^{-1}q^{D-2i+1}}, \\
  a_1 - z_i &=  \frac{a_1q^{1-i}}{1-aq^{1-D}} \, \frac{a^{-1}q^{D-i}-aq^{i-D}}{1+ a^{-1} q^{D-2i+1}}.
  \end{align*}
  In particular,
  \begin{align*}
  z_2 &=   \frac{a_1}{1-aq^{1-D}} \, \frac{1-q^{-2}}{1+ a^{-1} q^{D-3}}, \\
   a_1 - z_2 &=  \frac{a_1q^{-1}}{1-aq^{1-D}} \, \frac{a^{-1}q^{D-2}-aq^{2-D}}{1+ a^{-1} q^{D-3}}.
  \end{align*}
  \end{lemma}
  \begin{proof} Evaluate \eqref{eq:ziraw} using  \eqref{eq:vtheta} and the Appendix data.
  \end{proof}
  \begin{lemma} Assume that $a_1 \not=0$. Then $z_i$ and $a_1-z_i$ are nonzero for $2 \leq i \leq D$.
  \end{lemma} 
  \begin{proof} By Lemma \ref{lem:c7} and  \eqref{eq:ineq1}, \eqref{eq:ineq2}.
  \end{proof}
  
  \noindent We have a  comment.
  \begin{lemma} For $1 \leq i \leq D-1$,
  \begin{align*}
  (a_1-z_i)z_{i+1} = a_i z_2.
  \end{align*}
  \end{lemma}
  \begin{proof} Use Lemma \ref{lem:c7} and the Appendix data.
  \end{proof}
  
  
  \noindent Next, we consider  $p^i_{2,i-1}$ for $2 \leq i \leq D$.  By combinatorial counting or \cite[Lemma~4.1.7]{bcn},
\begin{align} \label{eq:pi2im1}
p^i_{2,i-1} = \frac{c_i(a_i+a_{i-1}-a_1)}{c_2} \qquad \qquad (2 \leq i \leq D).
\end{align}
Using the Appendix data, we obtain 
\begin{align*}
a_i + a_{i-1}-a_1 &= \frac{a(a+a^{-1})q^{D-2i+1}(q^{i-1}-q^{1-i})(aq^{i-D}-a^{-1}q^{D-i})}{(q-q^{-1})(a+q^{D-2i+3})(a+q^{D-2i-1})}
\\
&\quad  \times \frac{(q+q^{-1})(a+q^{D-1})(a+q^{-D-1})(aq^{2-D}-a^{-1}q^{D-2})}{(aq-a^{-1}q^{-1})(a-q^{1-D})(a+q^{D-3})}
\end{align*}
for $2 \leq i \leq D-1$ and
\begin{align*}
a_D &+ a_{D-1}-a_1\\
 &=  \frac{a(a^2-a^{-2})q^{1-D}(q^{D-1}-q^{1-D})(q+q^{-1})(a+q^{D-1})(aq^{2-D}-a^{-1}q^{D-2})}{(q-q^{-1})(a+q^{3-D})(aq-a^{-1}q^{-1})(a-q^{1-D})(a+q^{D-3})}.
\end{align*}
These formulas show that if  $a_1 \not=0$, then $a_i+a_{i-1}-a_1 \not=0$ for $2 \leq i \leq D$ and  $p^i_{2,i-1} \not=0$ for $2 \leq i \leq D$.
  \medskip
  
  \noindent Pick an integer $i$ $(2 \leq i \leq D)$ such that $p^i_{2,i-1} \not=0$, and pick $y,z \in X$ such that
  \begin{align*}
\partial(x,y) = i-1, \qquad \quad \partial(x,z) = i, \qquad \quad \partial (y,z)=2.
\end{align*}
We will consider 
\begin{align*}
 \bigl \vert \Gamma_{i-1} (x) \cap \Gamma(y) \cap \Gamma(z) \bigr \vert,
\qquad \quad  \bigl \vert \Gamma_{i} (x) \cap \Gamma(y) \cap \Gamma(z) \bigr \vert.
\end{align*}
By construction,
\begin{align}
c_2 = \bigl \vert \Gamma_{i-1} (x) \cap \Gamma(y) \cap \Gamma(z) \bigr \vert +
  \bigl \vert \Gamma_{i} (x) \cap \Gamma(y) \cap \Gamma(z) \bigr \vert. \label{eq:easystep}
  \end{align}

 \begin{lemma} \label{lem:c8}  Pick an integer $i$ $(2 \leq i \leq D)$ such that $p^i_{2,i-1} \not=0$, and pick $y,z \in X$ such that
\begin{align*}
\partial(x,y) = i-1, \qquad \quad \partial(x,z) = i, \qquad \quad \partial (y,z)=2.
\end{align*}
Then
\begin{align*}
 \bigl \vert \Gamma_{i-1} (x) \cap \Gamma(y) \cap \Gamma(z) \bigr \vert  \biggl(1+ \frac{\vartheta_{i-1}}{q+q^{-1}}\biggr) 
=  \bigl \vert \Gamma_{i} (x) \cap \Gamma(y) \cap \Gamma(z) \bigr \vert  \biggl(1+ \frac{\vartheta_{i}}{q+q^{-1}}\biggr).
\end{align*}
\end{lemma}
\begin{proof}  For the equation in Lemma \ref{lem:c2}, compute the $(y,z)$-entry of each side.
\end{proof}

\noindent The following result is essentially due to Caughman and Wolff \cite[Lemma~9.4]{CW}.
\begin{proposition} \label{lem:c11} {\rm (See \cite[Lemma~9.4]{CW}).} Pick an integer $i$ $(2 \leq i \leq D)$ such that $p^i_{2,i-1} \not=0$, and pick $y,z \in X$ such that
\begin{align}
\partial(x,y) = i-1, \qquad \quad \partial(x,z) = i, \qquad \quad \partial (y,z)=2.   \label{eq:fixyz}
\end{align}
Then
\begin{align*}
 \bigl \vert \Gamma_{i-1} (x) \cap \Gamma(y) \cap \Gamma(z) \bigr \vert &= c_2 \frac{q+q^{-1}+ \vartheta_{i}}{2(q+q^{-1})+\vartheta_{i-1}+\vartheta_i}, \\
  \bigl \vert \Gamma_{i} (x) \cap \Gamma(y) \cap \Gamma(z) \bigr \vert &= c_2 \frac{q+q^{-1}+ \vartheta_{i-1}}{2(q+q^{-1})+\vartheta_{i-1}+\vartheta_i}.
\end{align*}
\end{proposition}
\begin{proof} Combine \eqref{eq:easystep} with Lemma \ref{lem:c8},  and use Lemma \ref{lem:dnz}.
\end{proof}

\begin{remark}\rm
Referring to Proposition \ref{lem:c11}, the following quantities are independent of the choice of $x,y,z$:
\begin{align*}
 \bigl \vert \Gamma_{i-1} (x) \cap \Gamma(y) \cap \Gamma(z) \bigr \vert, \qquad \qquad
  \bigl \vert \Gamma_{i} (x) \cap \Gamma(y) \cap \Gamma(z) \bigr \vert.
  \end{align*}
\end{remark}

\noindent Next, we rephrase Proposition \ref{lem:c11} in various ways.

\begin{lemma} \label{lem:c12}
 Referring to Proposition \ref{lem:c11}, we have
\begin{align*}
 \bigl \vert \Gamma_{i-1} (x) \cap \Gamma(y) \cap \Gamma(z) \bigr \vert &= c_2 \frac{a_{i-1}-z_i}{a_i + a_{i-1}-a_1},\\
  \bigl \vert \Gamma_{i} (x) \cap \Gamma(y) \cap \Gamma(z) \bigr \vert &=  c_2 \frac{a_{i}-a_1+z_i}{a_i + a_{i-1}-a_1}.
\end{align*}
\end{lemma}
\begin{proof} Consider the equations in the lemma statement. For each equation, the quantity
on the left is the same for all ordered pairs $y,z$ that satisfy  \eqref{eq:fixyz}.   Therefore, this quantity must be equal to its average value. This average value is routinely computed by combinatorial counting, and given on the right-hand side of the equation. To illustrate, we perform the combinatorial counting for the first equation. For this equation, let $Q$ denote the quantity on the left. Let
 $N$ denote the number of 3-tuples $(y,w,z)$ of vertices such that $y,z$ satisfy  \eqref{eq:fixyz} and $w \in  \Gamma_{i-1} (x) \cap \Gamma(y) \cap \Gamma(z)$.
We compute $N$ in two ways. On one hand, there are $k_i$ choices for $z$, and for each choice there are $c_i$ choices for $w$, and for each choice there are $a_{i-1}-z_i$ choices for $y$. Therefore
$N = k_i c_i (a_{i-1}-z_i)$. On the other hand, there are $k_i p^i_{2,i-1}$ choices for the pair $(y,z)$, and for each choice  there are $Q$  choices
for $w$. Therefore $N= k_i p^i_{2,i-1} Q$. By these comments, $c_i (a_{i-1}-z_i) = p^i_{2,i-1} Q$. Solve this equation for $Q$ and evaluate the result using
\eqref{eq:pi2im1} to find that $Q$ is equal to the right-hand side of the first equation.
\end{proof}

\begin{lemma} \label{lem:c12a}
 Referring to Proposition \ref{lem:c11}, we have
\begin{align*}
 \bigl \vert \Gamma_{i-1} (x) \cap \Gamma(y) \cap \Gamma(z) \bigr \vert &=c_2 \frac{q}{q+q^{-1}} \frac{a+q^{D-2i-1}}{a+q^{D-2i+1}},\\
  \bigl \vert \Gamma_{i} (x) \cap \Gamma(y) \cap \Gamma(z) \bigr \vert &= c_2   \frac{q^{-1}}{q+q^{-1}} \frac{a+q^{D-2i+3}}{a+q^{D-2i+1}}.
\end{align*}
\end{lemma}
\begin{proof} Evaluate Proposition \ref{lem:c11} using \eqref{eq:vtheta} and Lemma   \ref{lem:dnz}.
\end{proof}

\noindent   As an aside, we mention a generalization of \cite[Lemma~9.6]{CW}(i).
\begin{lemma}\rm Assume that $a_1 \not=0$. Then for $2 \leq i \leq D-1$,
\begin{align*}
(q+q^{-1})^2 \,\frac{a_{i-1}-z_i}{a_i+a_{i-1}-a_1} \,\frac{a_{i+1}-a_1+z_{i+1}}{a_{i+1}+a_i-a_1}=1.
\end{align*}
\end{lemma}
\begin{proof} Use Lemmas \ref{lem:c12}, \ref{lem:c12a}.
\end{proof}

\noindent In Lemma \ref{lem:c2} we used $Z$ to obtain a matrix equation. Next, we use $Z$ to obtain another matrix equation.

\begin{lemma} \label{lem:c14} For $0 \leq i \leq D$ there is a linear equation with the following terms and coefficients:
 \begin{align*}
 0 = \qquad                       
 \hbox{
\begin{tabular}[t]{c|c}
{\rm term}& {\rm coefficient} 
 \\
 \hline
$E^*_i {\sf A} E^*_{i-1} {\sf A} E^*_i $ & $2\vartheta_i - \beta \vartheta_{i-1}$ \\
$E^*_i {\sf A} E^*_{i} {\sf A} E^*_i $ & $(q-q^{-1})(q^2-q^{-2})$ \\
$E^*_i {\sf A} E^*_{i+1} {\sf A} E^*_i $ & $2\vartheta_i - \beta \vartheta_{i+1}$ \\
$E^*_i {\sf A} E^*_i $ & $-(q-q^{-1})(q^2-q^{-2})(q+q^{-1} + \vartheta_i)$ \\
$E^*_i $ & $(q^2-q^{-2})^2\vartheta_i$ 
   \end{tabular}
   }
   \end{align*}

\noindent In the above table, we interpret $E^*_{-1}=0$ and $E^*_{D+1}=0$.
\end{lemma}
\begin{proof}
For the first equation in Proposition  \ref{prop:AWZ}, multiply each term on the left and right by $E^*_i$.
Evaluate the result using the following observations:
\begin{align*}
E^*_i {\sf A}^2 {\sf B} E^*_i &= \vartheta_i E^*_i {\sf A}^2 E^*_i = E^*_i {\sf B} {\sf A}^2 E^*_i,
\\
E^*_i {\sf A}^2 E^*_i &= E^*_i {\sf A} I {\sf A} E^*_i= \sum_{j=0}^D E^*_i {\sf A} E^*_j {\sf A} E^*_i \\
 &=  E^*_i {\sf A} E^*_{i-1} {\sf A} E^*_i
+E^*_i {\sf A} E^*_i {\sf A} E^*_i+E^*_i {\sf A} E^*_{i+1} {\sf A} E^*_i, \\
E^*_i {\sf A}{\sf B} {\sf A} E^*_i &= \sum_{j=0}^D \vartheta_j E^*_i {\sf A} E^*_j {\sf A} E^*_i \\
&= \vartheta_{i-1} E^*_i {\sf A} E^*_{i-1} {\sf A} E^*_i
+\vartheta_i E^*_i {\sf A} E^*_i {\sf A} E^*_i+\vartheta_{i+1}E^*_i {\sf A} E^*_{i+1} {\sf A} E^*_i,
\end{align*}
\begin{align*}
 E^*_i {\sf B} E^*_i &= \vartheta_i E^*_i, \\
 E^*_i Z E^*_i &= E^*_i Z = E^*_i {\sf A} E^*_i \biggl( 1 + \frac{\vartheta_i}{q+q^{-1}}\biggr), \\
 E^*_i Z {\sf A} E^*_i &= E^*_i {\sf A} E^*_i {\sf A} E^*_i \biggl( 1 + \frac{\vartheta_i}{q+q^{-1}}\biggr).
\end{align*}
\end{proof}

\noindent Next, we restate Lemma \ref{lem:c14} using $A$ instead of $\sf A$.
  
  \begin{lemma} \label{lem:c15} For $0 \leq i \leq D$ there is a linear equation with the following terms and coefficients:
 \begin{align*}
 0 = \qquad 
 \hbox{
\begin{tabular}[t]{c|c}
{\rm term}& {\rm coefficient} 
 \\
 \hline
$E^*_i A E^*_{i-1} A E^*_i $ & $\frac{2\vartheta_i - \beta \vartheta_{i-1}}{(q-q^{-1})(q^2-q^{-2})}$ \\
$E^*_i A E^*_{i} A E^*_i $ & $1$ \\
$E^*_i A E^*_{i+1} A E^*_i $ & $\frac{2\vartheta_i - \beta \vartheta_{i+1}}{(q-q^{-1})(q^2-q^{-2})}$ \\
$E^*_i A E^*_i $ & $-2\varepsilon - \alpha (q+q^{-1} + \vartheta_i)$ \\
$E^*_i $ & $\varepsilon^2 + \alpha \varepsilon (q+q^{-1}+ \vartheta_i) + (q+q^{-1})\alpha^2 \vartheta_i$ 
   \end{tabular}}
   \end{align*}
\noindent In the above table, we interpret $E^*_{-1}=0$ and $E^*_{D+1}=0$.
\end{lemma}
\begin{proof} Eliminate $\sf A$ in Lemma \ref{lem:c14} using \eqref{eq:AB}, and simplify the result.
\end{proof}

\noindent For $y,z \in X$ we will examine the $(y,z)$-entry in the linear equation from Lemma \ref{lem:c15}.  We will assume $y,z \in \Gamma_i(x)$ and $\partial(y,z) \leq 2$;
otherwise the $(y,z)$-entry of each term is zero. 

\begin{lemma} \label{lem:cibi} For $0 \leq i \leq D$,
\begin{align*}
0 &= c_i \, \frac{2 \vartheta_i - \beta \vartheta_{i-1}}{(q-q^{-1})(q^2-q^{-2})} + a_i + b_i \, \frac{2\vartheta_i - \beta \vartheta_{i+1}}{(q-q^{-1})(q^2-q^{-2})} \\
& \quad + \varepsilon^2 + \alpha \varepsilon (q+ q^{-1} + \vartheta_i) + (q+ q^{-1}) \alpha^2 \vartheta_i.
\end{align*} 
\end{lemma}
\begin{proof} Pick $y \in \Gamma_i(x)$. Compute the $(y,y)$-entry of each term in the linear equation from Lemma \ref{lem:c15}.
\end{proof}
\begin{remark}\rm The equation in Lemma \ref{lem:cibi} can be checked directly using  the Appendix data.
\end{remark}

\noindent The next lemma is essentially due to Curtin and Nomura \cite[Lemma~3.5]{CNhom}. We mentioned in the Appendix that $a_1\not=0$ implies
$a_i \not=0$ for $1 \leq i \leq D$.
\begin{lemma} \label{lem:CN} {\rm (See \cite[Lemma~3.5]{CNhom}).} 
Assume  that $a_1 \not=0$.
Pick an integer $i$ $(1 \leq i \leq D)$, and pick $y,z \in X$ such that
\begin{align} \label{eq:yzchoice}
\partial(x,y) = i, \qquad \quad \partial(x,z) = i, \qquad \quad \partial (y,z)=1.
\end{align}
If $1 \leq i \leq D-1$ then
\begin{align*}
 \bigl \vert \Gamma_{i-1} (x) \cap \Gamma(y) \cap \Gamma(z) \bigr \vert &= \frac{c_i(a_1-z_i)}{a_i},\\
  \bigl \vert \Gamma_i (x) \cap \Gamma(y) \cap \Gamma(z) \bigr \vert &=a_1 - \frac{c_i(a_1-z_i)}{a_i} -     \frac{b_iz_{i+1}}{a_i},\\
  \bigl \vert \Gamma_{i+1} (x) \cap \Gamma(y) \cap \Gamma(z) \bigr \vert &= \frac{b_iz_{i+1}}{a_i}.
  \end{align*}
  If $i=D$ then
  \begin{align*}
 \bigl \vert \Gamma_{D-1} (x) \cap \Gamma(y) \cap \Gamma(z) \bigr \vert &= \frac{c_D(a_1-z_D)}{a_D},\\
  \bigl \vert \Gamma_D(x) \cap \Gamma(y) \cap \Gamma(z) \bigr \vert &=a_1 - \frac{c_D(a_1-z_D)}{a_D}.
  \end{align*}
\end{lemma}
\begin{proof} Consider the equations in the lemma statement. It is shown in \cite[Lemma~3.5]{CNhom} that for each equation, the quantity
on the left is the same for all ordered pairs $y,z$ that satisfy \eqref{eq:yzchoice}.  Therefore, this quantity must be equal to its average value. This average value is routinely computed by combinatorial counting, and given on the right-hand side of the equation. We remark that the condition $t_1 \not\in \lbrace t_0, -t_0\rbrace$ given in \cite[Lemma~3.5]{CNhom} is satisfied.
The reason is that
$\tau_1 \not\in \lbrace \tau_0, -\tau_0\rbrace$
by  \eqref{eq:tFORM} and  \eqref{eq:ineq2}, and also $\tau_1/\tau_0 = t_0/t_1$ because  ${\sf W}= \sum_{i=0}^D t_i A_i$ by \cite[Section~3]{CNhom}.
\end{proof}

  \begin{lemma} \label{lem:c9} Assume that $a_1\not=0$.  For $1 \leq i \leq D-1$,
\begin{align*}
2\varepsilon +\alpha (q+q^{-1} + \vartheta_i) &=
 \frac{c_i(a_1-z_i)}{a_i}\, \frac{2\vartheta_i - \beta \vartheta_{i-1}}{(q-q^{-1})(q^2-q^{-2})} \\
 &\quad+ a_1 - \frac{c_i(a_1-z_i)}{a_i} -     \frac{b_iz_{i+1}}{a_i} \\
&\quad+ \frac{b_iz_{i+1}}{a_i}\, \frac{2\vartheta_i - \beta \vartheta_{i+1}}{(q-q^{-1})(q^2-q^{-2})}.
\end{align*}
\noindent Moreover,
\begin{align*}
2\varepsilon +\alpha (q+q^{-1} + \vartheta_D) &=
 \frac{c_D(a_1-z_D)}{a_D}\, \frac{2\vartheta_D - \beta \vartheta_{D-1}}{(q-q^{-1})(q^2-q^{-2})} \\
 &\quad+ a_1 - \frac{c_D(a_1-z_D)}{a_D}.
\end{align*}
\end{lemma}
\begin{proof} Let $1 \leq i \leq D$. Pick $y,z \in \Gamma_i(x)$ such that $\partial(y,z)=1$. For the linear equation in Lemma \ref{lem:c15},
compute the $(y,z)$-entry of each term using Lemma \ref{lem:CN}. The result follows.
\end{proof}

\begin{remark}\rm The equations in Lemma \ref{lem:c9} can be checked directly using Lemma \ref{lem:c7} and the Appendix data.
\end{remark}

\noindent Next, we consider $p^i_{2,i}$ for $2 \leq i \leq D-1$. By combinatorial counting or \cite[Lemma~4.1.7]{bcn},
\begin{align*}
p^i_{2,i} = \frac{ c_i(b_{i-1}-1) + b_i(c_{i+1}-1) + a_i(a_i-a_1-1)}{c_2} \qquad \qquad (2 \leq i \leq D-1).
\end{align*}
We will not express $p^i_{2,i}$ in terms of $a, q$ because in general the formula does not factor nicely.
The reader should keep in mind that possibly $p^i_{2,i} =0$.
\medskip

\noindent
Pick an integer $i$ $(2 \leq i \leq D-1)$ such that $p^i_{2,i} \not=0$. Pick $y,z \in X$ such that
\begin{align*}
\partial(x,y) = i, \qquad \quad \partial(x,z) = i, \qquad \quad \partial (y,z)=2.
\end{align*}
We will consider 
\begin{align*} 
 \bigl \vert \Gamma_{i-1} (x) \cap \Gamma(y) \cap \Gamma(z) \bigr \vert, \qquad 
   \bigl \vert \Gamma_{i} (x) \cap \Gamma(y) \cap \Gamma(z) \bigr \vert, \qquad 
     \bigl \vert \Gamma_{i+1} (x) \cap \Gamma(y) \cap \Gamma(z) \bigr \vert.
   \end{align*}
 By construction,
\begin{align} \label{eq:sumac}
c_2 =  \bigl \vert \Gamma_{i-1} (x) \cap \Gamma(y) \cap \Gamma(z) \bigr \vert+
   \bigl \vert \Gamma_{i} (x) \cap \Gamma(y) \cap \Gamma(z) \bigr \vert +
     \bigl \vert \Gamma_{i+1} (x) \cap \Gamma(y) \cap \Gamma(z) \bigr \vert.
   \end{align}

  \begin{lemma} \label{lem:c19} Pick an integer $i$ $(2 \leq i \leq D-1)$ such that $p^i_{2,i} \not=0$. Pick $y,z \in X$ such that
\begin{align*}
\partial(x,y) = i, \qquad \quad \partial(x,z) = i, \qquad \quad \partial (y,z)=2.
\end{align*}
Then
\begin{align*}
0 &=
 \vert \Gamma_{i-1}(x) \cap \Gamma(y) \cap \Gamma(z) \vert \frac{2\vartheta_i - \beta \vartheta_{i-1}}{(q-q^{-1})(q^2-q^{-2})} \\
 &\quad+ \vert \Gamma_{i}(x) \cap \Gamma(y) \cap \Gamma(z) \vert \\
&\quad+ \vert \Gamma_{i+1}(x) \cap \Gamma(y) \cap \Gamma(z) \vert \frac{2\vartheta_i - \beta \vartheta_{i+1}}{(q-q^{-1})(q^2-q^{-2})}.
\end{align*}
\end{lemma}
\begin{proof} For the linear equation in Lemma \ref{lem:c15}, compute the $(y,z)$-entry of each term.
\end{proof}

\begin{lemma}\label{lem:c20} Referring to Lemma \ref{lem:c19},
\begin{align*}
-c_2 &=
 \vert \Gamma_{i-1}(x) \cap \Gamma(y) \cap \Gamma(z) \vert \biggl( \frac{2\vartheta_i - \beta \vartheta_{i-1}}{(q-q^{-1})(q^2-q^{-2})}-1\biggr) \\
&\quad+ \vert \Gamma_{i+1}(x) \cap \Gamma(y) \cap \Gamma(z) \vert \biggl(  \frac{2\vartheta_i - \beta \vartheta_{i+1}}{(q-q^{-1}  )  (q^2-q^{-2})}-1 \biggr).
\end{align*}
\end{lemma}
\begin{proof} Compare  \eqref{eq:sumac} and Lemma \ref{lem:c19}.
\end{proof}

\noindent Let us examine the coefficients that show up in Lemma \ref{lem:c20}.  For $2 \leq i \leq D-1$ we have
\begin{align*}
& \frac{2\vartheta_i - \beta \vartheta_{i-1}}{(q-q^{-1})(q^2-q^{-2})}- 1 = \frac{q^{2i-D-2}(a+ q^{D-2i+1})(a-q^{D-2i+3})}{a(q-q^{-1})}.
\end{align*}
Call this common value $\xi_i$, and note that $\xi_i \not=0$ for $2 \leq i \leq D-1$.
For $2 \leq i \leq D-1$ we have
\begin{align*}
\frac{2\vartheta_i - \beta \vartheta_{i+1}}{(q-q^{-1} )   (q^2-q^{-2})}-1  = - \,\frac{q^{2i-D+2}(a+ q^{D-2i-1})(a-q^{D-2i-3})}{a(q-q^{-1})}.
\end{align*}
Call this common value $-\zeta_i$. Note that $\zeta_i \not=0$ for $2 \leq i \leq D-2$, and $\zeta_{D-1}=0$ if and only if $a=q^{-D-1}$.

\begin{proposition}\label{prop:Dconst} For $2 \leq i \leq D-1$ we have
\begin{align}\label{eq:DUineq}
0 &\leq \frac{\zeta_i }{\xi_i } \biggl( p^i_{2,i}(c_2-z_2-1) - (b_{i-1}-a_1-1+z_i)(c_{i+1}-z_{i+1}-1)    \biggr).
\end{align}
\end{proposition}
\begin{proof} 
Fix $y \in \Gamma_i(x)$. The set $\Gamma_i(x) \cap \Gamma_2(y)$ has cardinality $p^i_{2,i}$.
We define two functions $\mathcal D, \mathcal U$ on the set $\Gamma_i(x) \cap \Gamma_2(y)$. For $z \in \Gamma_i(x) \cap \Gamma_2(y)$, 
\begin{align*}
{\mathcal  D}(z) = \vert \Gamma_{i-1}(x) \cap \Gamma(y) \cap \Gamma(z)\vert, \qquad \quad {\mathcal U}(z) = \vert \Gamma_{i+1}(x) \cap \Gamma(y) \cap \Gamma(z)\vert.
\end{align*}
\noindent By combinatorial counting,
\begin{align*}
 \sum_{z \in \Gamma_i(x) \cap \Gamma_2(y)} {\mathcal D}(z) = c_i (b_{i-1}-a_1-1+z_i),
\qquad 
 \sum_{z \in \Gamma_i(x) \cap \Gamma_2(y)} {\mathcal U}(z) = b_i (c_{i+1}-z_{i+1}-1).
 \end{align*}
Assume for the moment that $p^i_{2,i} =0$. Then  $\Gamma_i(x) \cap \Gamma_2(y)$ is empty, so
\begin{align*}
b_{i-1}-a_1-1+z_i =0, \qquad \qquad c_{i+1}-z_{i+1}-1=0.
\end{align*}
In this case, \eqref{eq:DUineq} holds with equality. For the rest of this proof, we assume that $p^i_{2,i} \not=0$.
We consider the average values
\begin{align*}
{\mathcal D}_i &= \frac{1}{p^i_{2,i}} \sum_{z \in \Gamma_i(x) \cap \Gamma_2(y)} {\mathcal D}(z) = \frac{ c_i(b_{i-1}-a_1-1+z_i)}{p^i_{2,i}}, \\
{\mathcal U}_i &= \frac{1}{p^i_{2,i}} \sum_{z \in \Gamma_i(x) \cap \Gamma_2(y)} \mathcal U(z)  = \frac{b_i(c_{i+1}-z_{i+1}-1)}{p^i_{2,i}}.
\end{align*} 
By Lemma \ref{lem:c20} the following holds 
for $z \in \Gamma_i(x) \cap \Gamma_2(y)$:
\begin{align} \label{eq:c2need}
 - c_2 &= \mathcal D(z) \xi_i - \mathcal U(z) \zeta_i.
\end{align}
Summing \eqref{eq:c2need} over $z \in \Gamma_i(x) \cap \Gamma_2(y)$, we obtain
\begin{align}\label{eq:c2need2}
-c_2 &= \mathcal D_i \xi_i - \mathcal U_i \zeta_i.
\end{align}
Combining \eqref{eq:c2need}, \eqref{eq:c2need2} we find that for $z \in \Gamma_i(x) \cap \Gamma_2(y)$,
\begin{align*} 
(\mathcal D(z) - \mathcal D_i) \xi_i = (\mathcal U(z) - \mathcal U_i) \zeta_i.
\end{align*}
By combinatorial counting, 
\begin{align*}
\sum_{z \in \Gamma_i(x) \cap \Gamma_2(y)} \mathcal D(z) \mathcal U(z) = c_i b_i (c_2-z_2-1).
\end{align*}
We may now argue
\begin{align*}
0 &\leq  \sum_{z \in \Gamma_i(x) \cap \Gamma_2(y)} (\mathcal D(z)-\mathcal D_i)^2  \\
&=\frac{\zeta_i}{\xi_i} \sum_{z \in \Gamma_i(x) \cap \Gamma_2(y)} \bigl(\mathcal D(z)- \mathcal D_i \bigr) \bigl(\mathcal U(z) - \mathcal U_i \bigr) \\
&=\frac{\zeta_i}{\xi_i} \sum_{z \in \Gamma_i(x) \cap \Gamma_2(y)} \bigl(\mathcal D(z) \mathcal U(z) -  \mathcal D_i \mathcal U_i \bigr)  \\
&=\frac{\zeta_i c_i b_i }{\xi_i p^i_{2,i} } \biggl( p^i_{2,i}(c_2-z_2-1) - (b_{i-1}-a_1-1+z_i)(c_{i+1}-z_{i+1}-1)    \biggr).
\end{align*}
The result follows.
\end{proof}
\noindent Next, we consider the case of equality in \eqref{eq:DUineq}. If $p^i_{2,i} =0$  then equality holds in
\eqref{eq:DUineq}, according to the previous proof. We now consider the case in which $p^i_{2,i} \not=0$.

\begin{proposition} \label{prop:cond3} Pick an integer $i$ $(2 \leq i \leq D-1)$ such that $p^i_{2,i}\not=0$. Then the following are equivalent:
\begin{enumerate}
\item[\rm (i)] equality holds in \eqref{eq:DUineq};
\item[\rm (ii)] for all $y,z \in \Gamma_i(x)$ such that $\partial (y,z)=2$,
\begin{align*}
\vert \Gamma_{i-1}(x) \cap \Gamma(y) \cap \Gamma(z) \vert = \frac{ c_i (b_{i-1}-a_1-1+z_i)}{p^i_{2,i}};
\end{align*}
\item[\rm (iii)] there exists $y \in \Gamma_i(x)$ such that
$\vert \Gamma_{i-1}(x) \cap \Gamma(y) \cap \Gamma(z) \vert $ is the same for all $z \in \Gamma_i(x) \cap \Gamma_2(y)$.
\end{enumerate}
\end{proposition}
\begin{proof} Immediate from the proof of Proposition \ref{prop:Dconst}.
\end{proof}
\noindent Next, we express the right-hand side of \eqref{eq:DUineq} in terms of $a, q$. By the discussion below
Lemma \ref{lem:c20},

\begin{align*}
\frac{\zeta_i}{\xi_i} = \frac{ q^4 (a+q^{D-2i-1})(a-q^{D-2i-3})}{(a+ q^{D-2i+1})(a-q^{D-2i+3})} \qquad \qquad (2 \leq i \leq D-1).
\end{align*}

 \begin{proposition}\label{prop:bigform} For $2 \leq i \leq D-1$,
 \begin{align*}
& p^i_{2,i}(c_2-z_2-1) - (b_{i-1}-a_1-1+z_i)(c_{i+1}-z_{i+1}-1)  
\\
&= \frac{(q^i-q^{-i})(q^{i-1}-q^{1-i})  (aq^{i-D}-a^{-1} q^{D-i})(a-q^{D-2i-3})(a-q^{D-2i+3})(aq^{i-D+1}-a^{-1}q^{D-i-1})         }
 {(aq^{2i-D-2}-a^{-1} q^{D+2-2i})(aq^{2i+2-D}-a^{-1} q^{D-2i-2})  (a+ q^{D-2i+1})(a+q^{D-2i-1})   } 
 \\
&\quad \times \frac{q^{2-2D}(a+a^{-1})(a^2 q-a^{-2}q^{-1})(aq^D-a^{-1} q^{-D})(a q^{2-D}-a^{-1} q^{D-2})}
 {(aq-a^{-1}q^{-1})^2 (q-q^{-1})^2 (q+q^{-1})(aq^2-a^{-1}q^{-2})} 
\\
&\quad \times  \frac{(a+q^{-D-1})(a-q^{D+1})(a+q^{D-1})^2}
 { (a+q^{D-3}) (a-q^{3-D})(a-q^{1-D})^2}. 
 \end{align*}
 \end{proposition}
 \begin{proof} Use Lemma  \ref{lem:c7} 
 and the Appendix data.
 \end{proof}
 
 
 \begin{remark}\rm \label{remark} 
 Assume that 
\begin{align}\label{eq:aaaa}
a_1 =0 \quad {\rm or} \quad  a=q^{D+1} \quad {\rm or}  \quad a^2=q^{-2D} \quad {\rm or} \quad a^4=q^{-2}. 
\end{align}
 Then the equivalent conditions (i)--(iii)
 in Proposition \ref{prop:cond3} hold for $2 \leq i \leq D-1$. Next assume that $a=q^{-D-1}$, and \eqref{eq:aaaa} does not hold. Then the equivalent conditions (i)--(iii) 
 in Proposition \ref{prop:cond3} hold for $i= D-1$ but not for $2 \leq i \leq D-2$.
 Next assume that $a\not=q^{-D-1}$, and \eqref{eq:aaaa} does not hold.
 Then for $2 \leq i \leq D-1$  the equivalent conditions (i)--(iii) 
 in Proposition \ref{prop:cond3} do not hold.
 \end{remark} 

\noindent Next we consider $p^D_{2,D}$. By combinatorial counting or \cite[Lemma~4.1.7]{bcn},
\begin{align*}
p^D_{2,D} = \frac{ c_D(b_{D-1}-1) + a_D(a_D-a_1-1)}{c_2}.
\end{align*}
\noindent Using the Appendix data, we obtain
\begin{align*}
p^D_{2,D} & = \frac{ (q^D-q^{-D})(q^{D-1}-q^{1-D})(a^2-a^{-2})(a^2 q - a^{-2} q^{-1})}
{(q^2-q^{-2})(q-q^{-1}) (aq-a^{-1}q^{-1})(aq^2-a^{-1}q^{-2})} \\
&\quad \times \frac{ (aq^{1-D}-a^{-1}q^{D-1})(aq^{4-D}-a^{-1} q^{D-4})}
{ (aq^{D-1}-a^{-1} q^{1-D})(aq^{D-2}-a^{-1} q^{2-D})}.
\end{align*}
Note that $p^D_{2,D}=0$ if and only if $a^2=-1$ or $a^4 = q^{-2}$. 
\medskip

\noindent
For the moment, assume  that $p^D_{2,D}\not=0$.  Pick $y,z \in X$ such that
\begin{align*}
\partial(x,y) = D, \qquad \quad \partial(x,z) = D, \qquad \quad \partial (y,z)=2.
\end{align*}
We will consider 
\begin{align} \label{eq:unknownD}
 \bigl \vert \Gamma_{D-1} (x) \cap \Gamma(y) \cap \Gamma(z) \bigr \vert, \qquad 
   \bigl \vert \Gamma_{D} (x) \cap \Gamma(y) \cap \Gamma(z) \bigr \vert.
   \end{align}
 By construction,
\begin{align} \label{eq:sumacd}
c_2 =  \bigl \vert \Gamma_{D-1} (x) \cap \Gamma(y) \cap \Gamma(z) \bigr \vert+
   \bigl \vert \Gamma_{D} (x) \cap \Gamma(y) \cap \Gamma(z) \bigr \vert.
   \end{align}

  \begin{lemma} \label{lem:c19d} Assume that $p^D_{2,D}\not=0$.  Pick $y,z \in X$ such that
\begin{align*}
\partial(x,y) = D, \qquad \quad \partial(x,z) = D, \qquad \quad \partial (y,z)=2.
\end{align*}
Then
\begin{align} \label{eq:missing}
0 &=
 \vert \Gamma_{D-1}(x) \cap \Gamma(y) \cap \Gamma(z) \vert \frac{2\vartheta_D - \beta \vartheta_{D-1}}{(q-q^{-1})(q^2-q^{-2})} + \vert \Gamma_{D}(x) \cap \Gamma(y) \cap \Gamma(z) \vert.
\end{align}
\end{lemma}
\begin{proof} For the linear equation in Lemma \ref{lem:c15}, compute the $(y,z)$-entry of each term.
\end{proof}

\noindent In a moment, we will solve the linear equations \eqref{eq:sumacd}, \eqref{eq:missing} for the unknowns \eqref{eq:unknownD}.  Note that
\begin{align*}
& \frac{2\vartheta_D - \beta \vartheta_{D-1}}{(q-q^{-1})(q^2-q^{-2})}- 1 = \frac{q^{D-2}(a+ q^{1-D})(a-q^{3-D})}{a(q-q^{-1})}.
\end{align*}
This common value is nonzero by \eqref{eq:ineq2}.
\medskip

 \noindent The following result is essentially due to Curtin and Nomura \cite[Lemma~3.5]{CNhom}.

\begin{proposition} \label{lem:Gform} {\rm (See \cite[Lemma~3.5]{CNhom}).}    Assume that $p^D_{2,D}\not=0$. Pick $y,z \in X$ such that
\begin{align}
\partial(x,y) = D, \qquad \quad \partial(x,z) = D, \qquad \quad \partial (y,z)=2.       \label{eq:xyzDD}
\end{align}
Then
\begin{align*}
\vert \Gamma_{D-1}(x) \cap \Gamma(y) \cap \Gamma(z) \vert  &= \frac{c_2 (q-q^{-1})(q^2-q^{-2})}{(q-q^{-1})(q^2-q^{-2}) - 2\vartheta_D + \beta \vartheta_{D-1}},\\
\vert \Gamma_D(x) \cap \Gamma(y) \cap \Gamma(z) \vert &= \frac{c_2 (\beta \vartheta_{D-1} - 2 \vartheta_D)}{(q-q^{-1})(q^2-q^{-2})-2\vartheta_D+\beta \vartheta_{D-1}}.
\end{align*}
\end{proposition}
\begin{proof} Solve the linear equations  \eqref{eq:sumacd}, \eqref{eq:missing} using the comment above the proposition statement.
\end{proof}

\begin{remark}\rm
Referring to Proposition \ref{lem:Gform}, the following quantities are independent of the choice of $x,y,z$:
\begin{align*}
 \bigl \vert \Gamma_{D-1} (x) \cap \Gamma(y) \cap \Gamma(z) \bigr \vert, \qquad \qquad
  \bigl \vert \Gamma_{D} (x) \cap \Gamma(y) \cap \Gamma(z) \bigr \vert.
  \end{align*}
\end{remark}

\noindent Next, we rephrase Proposition \ref{lem:Gform} in various ways.

\begin{lemma} \label{lem:Gformd} Referring to Proposition \ref{lem:Gform}, we have
\begin{align*}
\vert \Gamma_{D-1}(x) \cap \Gamma(y) \cap \Gamma(z) \vert  &= \frac{c_D(b_{D-1}-a_1-1+z_D)}{p^D_{2,D}},\\
\vert \Gamma_D(x) \cap \Gamma(y) \cap \Gamma(z) \vert &= \frac{c_D(a_1-z_D)+a_D(a_D-a_1-1)}{p^D_{2,D}}.
\end{align*}
\end{lemma}
\begin{proof}
Consider the equations in the lemma statement. For each equation, the quantity
on the left is the same for all ordered pairs $y,z$ that satisfy \eqref{eq:xyzDD}.   Therefore, this quantity must be equal to its average value. This average value is routinely computed by combinatorial counting, and given on the right-hand side of the equation.
\end{proof}

\begin{lemma} \label{lem:versionsd} Referring to Proposition \ref{lem:Gform}, we have
\begin{align*}
&\vert \Gamma_{D-1}(x) \cap \Gamma(y) \cap \Gamma(z) \vert \\
&\quad = 
\frac{-q^{-1} (q^2-q^{-2})(aq^2-a^{-1}q^{-2})(aq^{2-D}-a^{-1}q^{D-2})(a+q^{D-1})}
{(aq-a^{-1} q^{-1})(aq^{D-1}-a^{-1} q^{1-D})(aq^{4-D}-a^{-1} q^{D-4})(a+ q^{D-3})},
\\
&\vert \Gamma_D(x) \cap \Gamma(y) \cap \Gamma(z) \vert \\
&\quad =\frac{ q^{-1} (q+q^{-1})(aq^2 -a^{-1} q^{-2})(aq^{D-2}-a^{-1} q^{2-D})(aq^{2-D}-a^{-1} q^{D-2})(a+ q^{D-1})}
{ (aq-a^{-1} q^{-1})(aq^{D-1}-a^{-1} q^{1-D})(aq^{4-D}-a^{-1} q^{D-4})(a+ q^{D-3})}.
\end{align*}
\end{lemma}
\begin{proof} Evaluate the formulas in Proposition \ref{lem:Gform}, using  \eqref{eq:vtheta} and the  Appendix data.
\end{proof}

 \noindent By the {\it local graph} at $x$, we mean the subgraph of $\Gamma$ induced by $\Gamma(x)$. 
 For the sake of completeness, we mention a result about this local graph due to Curtin and Nomura \cite[Theorem~4.2]{CNhom}. We will translate their result into our notation (see Remarks \ref{rem:notation}, \ref{rem:typo}) and point of view,
 using standard results from \cite[Theorem~1.3.1]{bcn}.
  
\begin{lemma} \label{lem:c10} {\rm (See \cite[Theorem~4.2]{CNhom}).}     Assume that $a_1 \not=0$. Then the local graph $\Gamma(x)$ is connected strongly-regular, with eigenvalues $a_1$ and
\begin{align*}
r &= \frac{a (a+a^{-1})(aq^{2-D}-a^{-1}q^{D-2})}{q(a-q^{1-D})(a+q^{D-3})}, \\
s&= \frac{(1+aq^{D+1})(a^{-1}q^{D-2}-aq^{2-D})}{q^2(aq-a^{-1}q^{-1})(a+q^{D-3})}.
\end{align*}
The multiplicities of $r,s$ are
\begin{align*}
mult_r &= \frac{(q^{D-1}-q^{1-D})(1-aq^{1-D})(1+aq^{D+1})(a^3-q^{D-1})}{a(1-a^3q^{D+1})(q-q^{-1})(aq-a^{-1}q^{-1})}, \\
mult_s &= \frac{q^{D+1} (a+a^{-1})(q^{-D}-q^{D})(1-aq^{1-D})(a^3-q^{D-3})}{(q-q^{-1})(1-aq^{D-1})(1-a^3q^{D+1})}.
\end{align*}
\end{lemma}  
 \begin{remark} \label{rem:typo} \rm There is a typo in \cite[Theorem~4.2]{CNhom}. The given value of $\mu$ is off by a factor of $-1$.
 \end{remark}

 \section{Reversing the logical direction} 

  \noindent In Section 14, we were given a distance-regular graph $\Gamma$ that satisfies various assumptions, and
  we showed that $\Gamma$ affords a spin model. In the present section, we reverse the logical direction as follows.
  \medskip
  
  \noindent
  For the rest of  this section, the following assumptions are in effect.
  Let $\Gamma=(X,\mathcal R)$ denote a distance-regular graph with diameter $D \geq 3$.
  We assume that $\Gamma$ affords a spin model $\sf W$. 
  By  \cite[Lemma~11.4]{nomSpinModel}, there exists an ordering $\lbrace E_i \rbrace_{i=0}^D$ of
  the primitive idempotents of $\Gamma$ that is formally self-dual; we assume that this ordering has $q$-Racah type. 
    Since $\lbrace E_i\rbrace_{i=0}^D$ is a basis for $M$ and
  $\sf W$ is an invertible element of $M$, there exist nonzero scalars $f$, $\lbrace \tau_i \rbrace_{i=0}^D$ in $\mathbb C$ such that $\tau_0=1$ and
  \begin{align*}
  {\sf W} = f \sum_{i=0}^D \tau_i E_i.
  \end{align*}
   By Lemma \ref{lem:WminusSM} the matrix ${\sf W}^{(-)}$ is a spin model. By  \cite[Lemma~11.2]{nomSpinModel} the spin model ${\sf W}^{(-)}$ is afforded by $\Gamma$.
  By  \cite[Section~12]{nomSpinModel} we see that after replacing $\sf W$ by $\sf W^{(-)}$ if necessary, there exist  nonzero $a, \alpha \in \mathbb C$ and $\varepsilon \in \mathbb C$
  that satisfy
  \begin{align}
  \tau_i &= (-1)^i a^{-i} q^{i(D-i)} \qquad \qquad (0 \leq i \leq D),       \label{eq:tFORM}
  \\
   \label{eq:eFORM2}
  \theta_i &= \theta^*_i=\alpha (aq^{2i-D}+a^{-1} q^{D-2i}) + \varepsilon \qquad  (0 \leq i \leq D)
  \end{align}
  and also \eqref{eq:ineq1}--\eqref{eq:bformx} in the Appendix.
 By construction Assumption \ref{def:spintype} is satisfied, and  \eqref{eq:Gtheta} matches  \eqref{eq:eFORM2}.
We will show that  $a, \alpha, \varepsilon$ satisfy Assumption \ref{def:spintype2}.
\medskip

\noindent There is a kind of type II matrix said to have Hadamard type \cite[Definition~8.5]{nomSpinModel}.
 By \cite[Lemma~8.4]{nomSpinModel} and \cite[Definition~8.5]{nomSpinModel}, the spin model $\sf W$ has Hadamard type if and only if $\tau_i \in \lbrace 1,-1\rbrace$ for $0 \leq i \leq D$.
In fact $\sf W$ does not have Hadamard type, because $\tau^2_1 \tau^{-1}_2 =q^2$ by \eqref{eq:tFORM} and $q^2 \not\in \lbrace 1,-1\rbrace$  by \eqref{eq:ineq1}. 
Pick any $x \in X$ and write $T=T(x)$. We will speak of the $T$-modules, using the definitions and notation from \cite{cerzo} and \cite[Section~9]{nomSpinModel}.
By \cite[Lemma~9.6]{nomSpinModel}, the standard $T$-module is a direct sum of irreducible $T$-modules. Let $U$ denote an irreducible $T$-module.
The $T$-module $U$ is thin,
by \cite[Proposition~11.10]{nomSpinModel} and because $\sf W$ does not have Hadamard type.
Let $r$ and $d$ denote the endpoint and diameter of $U$, respectively.
 By \cite[Lemma~12.6]{nomSpinModel}, the $T$-module $U$ has dual endpoint  $r$.
According to \cite[Lemma~12.9]{nomSpinModel} the intersection numbers of $U$ are given by
\begin{align*}
b_i(U)  &=
 \frac{ \alpha (q^{i-d} - q^{d-i}) (a q^{2r+i -D} - a^{-1} q^{D-2r-i})(a^3-q^{3D-2d-6r-2i-1}) }
        {a q^{D-d-2r} (a q^{2r+2i-D} - a^{-1} q^{D-2r-2i} ) (a + q^{D-2r-2i-1} ) },        
\\
c_i(U) &=
 \frac{ \alpha a (q^{i}- q^{-i})(a q^{d+2r+i-D} - a^{-1} q^{D-d-2r-i})
                  (a^{-1} - q^{2d-D+2r-2i+1}) }
        {q^{d-D+2r} (a q^{2r+2i-D} - a^{-1} q^{D-2r-2i}) (a + q^{D-2r-2i+1} ) }     
\end{align*}
for $1 \leq i \leq d-1$ and
\begin{align*}
b_0(U)&=
  \frac{ \alpha (q^{-d} - q^d) ( a^3 - q^{3D-2d-6r-1} ) }
         {a q^{D-d-2r} (a + q^{D-2r-1}) },        
\\
c_d(U) &=
 \frac{ \alpha  (q^{-d} - q^d)(a - q^{D-2r-1}) }
        { q^{d-1} (a + q^{D-2d-2r+1}) }                 
\end{align*}
and
\begin{align*}
 a_i(U) = \theta_r - c_i(U) - b_i(U) \qquad \qquad (0 \leq i \leq d),
 \end{align*}
 where $c_0(U)=0$ and $b_d(U)=0$. By \cite[Lemma~12.6]{nomSpinModel} the dual intersection numbers of $U$ satisfy
 \begin{align*}
 b^*_i(U)&= b_i(U) \quad (0 \leq i \leq d-1), \qquad \quad
 c^*_i(U) =c_i(U) \quad (1 \leq i \leq d), \\
 &\qquad \qquad  a^*_i(U)=a_i(U) \quad (0 \leq i \leq d).
 \end{align*}
 \noindent Using  the above formulas, we obtain
 \begin{align*}
&\frac{a_i(U)-\varepsilon}{\alpha} \biggl(1+ \frac{\theta_{r+i}-\varepsilon}{\alpha}\,\frac{1}{q+q^{-1}}\biggr) \\
& = \frac{(aq^{2r+d-D}+a^{-1}q^{D-d-2r})(q^{d+1}+q^{-d-1}) + (aq^{2r+d-D}+a^{-1}q^{D-d-2r})^2}{q+q^{-1}} \\
& =\frac{a^*_i(U)-\varepsilon}{\alpha} \biggl(1+ \frac{\theta_{r+i}-\varepsilon}{\alpha}\,\frac{1}{q+q^{-1}}\biggr) 
 \end{align*}
 for $0 \leq i \leq d$. Let $A$ denote the adjacency matrix of $\Gamma$. let $A^*=A^*(x)$ denote the dual adjacency matrix of $\Gamma$
 with respect to $x$ and the ordering $\lbrace E_i \rbrace_{i=0}^D$. The previous equations show that the following holds on $U$:
 \begin{align*}
& \sum_{i=0}^D E^*_i \frac{A -\varepsilon I}{\alpha} E^*_i  \biggl(1+ \frac{\theta_{i}-\varepsilon}{\alpha}\,\frac{1}{q+q^{-1}}\biggr) \\
& = \frac{(aq^{2r+d-D}+a^{-1}q^{D-d-2r})(q^{d+1}+q^{-d-1}) + (aq^{2r+d-D}+a^{-1}q^{D-d-2r})^2}{q+q^{-1}} I \\
& =\sum_{i=0}^D E_i \frac{A^* -\varepsilon I}{\alpha} E_i  \biggl(1+ \frac{\theta_{i}-\varepsilon}{\alpha}\,\frac{1}{q+q^{-1}}\biggr).
\end{align*}
We must have
  \begin{align*}
& \sum_{i=0}^D E^*_i \frac{A -\varepsilon I}{\alpha} E^*_i  \biggl(1+ \frac{\theta_{i}-\varepsilon}{\alpha}\,\frac{1}{q+q^{-1}}\biggr) \\
& =
  \sum_{i=0}^D E_i \frac{A^* -\varepsilon I}{\alpha} E_i  \biggl(1+ \frac{\theta_{i}-\varepsilon}{\alpha}\,\frac{1}{q+q^{-1}}\biggr),
\end{align*}
because the two sides agree on each irreducible $T$-module. By these comments, 
the scalars $a, \alpha, \varepsilon$ satisfy Assumption \ref{def:spintype2}.

  \section{Directions for future research}
  
  \noindent In this section, we give some  open problems.
  
  \begin{problem}\rm We refer to a distance-regular graph $\Gamma$ that satisfies  Assumptions \ref{def:spintype}, \ref{def:spintype2}.
  The quantities computed in Lemmas \ref{lem:c7}, \ref{lem:c12a}, \ref{lem:CN}, \ref{lem:versionsd} and Proposition \ref{prop:bigform}
  must be nonnegative integers, and this condition imposes a restriction on the parameters $q, a$. It is possible that this restriction implies an upper bound on the diameter $D$.
  Investigate this possibility by considering the cases $D=3, 4,5,\ldots $ in detail.
  \end{problem}
  
  \begin{problem}\rm Investigate the distance-regular graphs $\Gamma$ that satisfy Assumptions \ref{def:spintype}, \ref{def:spintype2} and at least one of
  \begin{align*}
  a = q^{D+1}, \qquad \quad a^2 = q^{-2D}, \qquad \quad a^4 = q^{-2}, \qquad \quad a=q^{-D-1}. 
  \end{align*}
  These cases are 
  discussed in Remark \ref{remark}.
  \end{problem}
  


  \noindent The following problem is motivated by Assumption \ref{def:spintype2}.
  \begin{problem}\rm  Let $\Gamma=(X, \mathcal R)$ denote a distance-regular graph with diameter $D\geq 3$ and adjacency matrix $A$.
   Fix $x \in X$ and write $T=T(x)$.
Assume that $T$ contains a central element $Z$ of the form
\begin{align*}
Z = \sum_{i=0}^D E^*_i A E^*_i \alpha_i + \sum_{i=0}^D E^*_i \beta_i, \qquad \qquad \alpha_i, \beta_i \in \mathbb C.
\end{align*}
Further assume that $Z$ is not a scalar multiple of the identity $I$.
 Investigate the algebraic and combinatorial consequences of these assumptions.
\end{problem}

  \noindent The following problem is motivated by \eqref{eq:GWWs}.
  \begin{problem}\rm Let $\Gamma=(X, \mathcal R)$ denote a distance-regular graph with diameter $D\geq 3$ and adjacency matrix $A$.
  Assume that $\Gamma$ is $Q$-polynomial with respect to the ordering $\lbrace E_i \rbrace_{i=0}^D$ of the primitive idempotents.
  Fix $x \in X$ and consider the dual adjacency matrix  $A^*=A^*(x)$ with respect to $\lbrace E_i \rbrace_{i=0}^D$.
  Assume that the Bose-Mesner algebra $M$ contains an invertible matrix $W$, and the dual Bose-Mesner algebra $M^*=M^*(x)$ 
  contains an invertble matrix $W^*$, such that 
   $W^{-1}A^*W= W^* A (W^*)^{-1}$. Investigate the algebraic and combinatorial consequences of these assumptions.
  \end{problem}

 \noindent The following problem is motivated by Section 15.
\begin{problem} \rm  Let $\Gamma=(X, \mathcal R)$ denote a distance-regular graph with diameter $D\geq 3$.
   Fix $x \in X$ and write $T=T(x)$. We make three assumptions (i)--(iii).
   \begin{enumerate}
   \item[\rm (i)] For $1 \leq i \leq D$ the following matrices are linearly dependent:
   \begin{align*}
   E^*_{i-1} A E^*_{i-1} A E^*_i, \qquad
    E^*_{i-1} A E^*_i A E^*_i, \qquad
     E^*_{i-1} A E^*_i.
     \end{align*}
     \item[\rm (ii)] For $2 \leq i \leq D-1$ the following matrices are linearly dependent:
   \begin{align*}
   E^*_{i} A E^*_{i-1} A E^*_i, \qquad
    E^*_{i} A E^*_i A E^*_i, \qquad
     E^*_{i} A E^*_{i+1} A E^*_i, \qquad 
     E^*_i A E^*_i, \qquad 
     E^*_i.
     \end{align*}
  \item[\rm (iii)] The following matrices are linearly dependent:
   \begin{align*}
   E^*_D A E^*_{D-1} A E^*_D, \qquad
    E^*_{D} A E^*_D A E^*_D, \qquad
     E^*_D A E^*_D, \qquad 
     E^*_D.
     \end{align*}
     \end{enumerate}
Investigate the algebraic and combinatorial consequences of the above three assumptions.
 \end{problem}
    \noindent The following problem is motivated by Section 16.
\begin{problem} \rm  Let $\Gamma=(X, \mathcal R)$ denote a distance-regular graph with diameter $D\geq 3$.
   Fix $x \in X$ and write $T=T(x)$. Assume that every irreducible $T$-module is thin. 
Further assume that up to isomorphism, there exists at most one irreducible $T$-module
 with any given endpoint and diameter. Investigate the algebraic and combinatorial consequences of these assumptions.
 \end{problem}
  
  \begin{conjecture}\rm The article \cite{nortonPT} discusses the Norton algebra of a distance-regular graph  $\Gamma = (X, \mathcal R)$ that is $Q$-polynomial
  with respect to a primitive idempotent $E$. Let $x,y\in X$ be distinct.  In \cite[Theorem~4.4]{nortonPT} the Norton algebra product of $E{\hat x} \star E{\hat y}$
  is expressed as a linear combination of three vectors
  \begin{align} \label{eq:threevec}
  C(x,y), \qquad B(x,y), \qquad E{\hat x} + E{\hat y}.
  \end{align}
  We conjecture that if $\Gamma$ satisfies Assumptions \ref{def:spintype}, \ref{def:spintype2} then the vectors \eqref{eq:threevec} are linearly dependent. 
  \end{conjecture}

  \section{Appendix}
  This appendix contains some data about the distance-regular graph $\Gamma=(X,\mathcal R)$ that satisfies Assumptions \ref{def:spintype}, \ref{def:spintype2}. 
  The scalars $a, \alpha, \varepsilon, q$ found below are from \eqref{eq:Gtheta}, \eqref{eq:MainAssume2}. 
   The data in this appendix is taken from (or easily follows from)  \cite[Section~12]{nomSpinModel}.
  
   \begin{align}
  q^{2i} &\neq 1 && (1 \leq i \leq D),                \label{eq:ineq1}
\\
  a^2 q^{2i} &\neq 1  &&  (1-D \leq i \leq D-1),             \label{eq:ineq2}
\\
 a^3 q^{2i-D-1} &\neq 1  &&  (1 \leq i \leq D),            \label{eq:ineq3}
\end{align}
  \begin{align}
 \alpha &= 
  \frac{(a q^{2-D} - a^{-1} q^{D-2}) ( a + q^{D-1}) }
         {q^{D-1} (q^{-1} - q)(a q - a^{-1} q^{-1}) (a - q^{1-D}) },       \label{eq:aformx}
\\
 \varepsilon &=
   \frac{q (a+a^{-1}) (a + q^{-D-1}) (a q^{2-D}- a^{-1} q^{D-2}) }
            { (q-q^{-1}) (a - q^{1-D}) (a q - a^{-1} q^{-1}) }.                \label{eq:bformx}
\end{align}

   \begin{align*}
 \vartheta_i &= a q^{2i-D} + a^{-1} q^{D-2i} \qquad \qquad (0 \leq i \leq D),
\\
  \theta_i &= \theta^*_i=\alpha \vartheta_i + \varepsilon \qquad \qquad (0 \leq i \leq D),
  \end{align*}
   \begin{align*}
b_0 = b^*_0 &=
  \frac{ \alpha (q^{-D} - q^D)(a^3 - q^{D-1}) }
         {a (a + q^{D-1}) },                          
\\
b_i = b^*_i &=
  \frac{ \alpha (q^{i-D} - q^{D-i})(a q^{i-D} - a^{-1} q^{D-i}) (a^3 - q^{D-2i-1}) }
         {a (a q^{2i-D} - a^{-1} q^{D-2i}) (a + q^{D-2i-1}) }    && (1 \leq i \leq D-1),  
\\
c_i =c^*_i &=
 \frac{\alpha a (q^{i} - q^{-i})(a q^i - a^{-1} q^{-i}) (a^{-1} - q^{D-2i+1} ) }
        { (a q^{2i-D} - a^{-1} q^{D-2i}) (a + q^{D-2i+1} ) } &&  (1 \leq i \leq D-1),   
\\
c_D =c^*_D&=
 \frac{ \alpha (q^{-D} - q^D) (a - q^{D-1}) }
        {q^{D-1} (a + q^{1-D}) },                       
        \\
a_i = a^*_i &= \frac{\alpha a (a+a^{-1})(1+aq^{D+1})(q^i-q^{-i})(a^{-1}q^{D-i}-aq^{i-D})}{q^{2i-D+1}(a+q^{D-1})(a+q^{D-2i-1})(a+q^{D-2i+1})}
&& (1 \leq i \leq D-1), \\
a_D = a^*_D &= \frac{\alpha a (a^{-2}-a^2)(q^D-q^{-D})}{(a+q^{D-1})(a+q^{1-D})}.
\end{align*}

\noindent
The following formulas and statements are readily checked.
\begin{align*}
\varepsilon &= a_1 \frac{q (a+q^{D-3})}{(q-q^{-1})(a-q^{D-1})}, \\
a_1 &= \frac{(a+a^{-1})(1-aq^{1-D})(1+aq^{D+1})(aq^{2-D}-a^{-1}q^{D-2})}{(1+aq^{3-D})(1-aq^{D-1})(aq-a^{-1}q^{-1})}.
\end{align*}

\noindent  For $1 \leq i \leq D-1$,
 \begin{align*}
 a_i = a_1\frac{a q^{2-2i}(a+q^{D-3}) (q^i-q^{-i})(a^{-1} q^{D-i}-aq^{i-D})}{(1-aq^{1-D})(q-q^{-1})(a+q^{D-2i+1})(a+q^{D-2i-1})}.    \end{align*}
For $1 \leq i \leq D$,
  \begin{align}  \label{eq:aiApp}
a_{i-1} \biggl(1+ \frac{\vartheta_{i-1}}{q+q^{-1}}\biggr) = a_i \biggl(1+ \frac{\vartheta_i}{q+q^{-1}}\biggr) + \frac{\vartheta_{i-1} - \vartheta_i}{q+q^{-1}} \varepsilon.
\end{align}
For $0 \leq i \leq D$,
\begin{align}
a_i \biggl(1+ \frac{\vartheta_i}{q+q^{-1}}\biggr) = \varepsilon \, \frac{\vartheta_i - \vartheta_0}{q+q^{-1}}.  \label{eq:aiApp2}
\end{align}

\noindent The graph $\Gamma$ is called {\it bipartite} whenever $a_i = 0$ for $0 \leq i \leq D$. This occurs if and only if $a^2=-1$.
The graph $\Gamma$ is called {\it almost bipartite} whenever $a_i = 0 $ for $0 \leq i \leq D-1$ and $a_D\not=0$. This occurs if and only if $a=-q^{-D-1}$.
If $\Gamma$ is neither bipartite nor almost bipartite, then $a_i \not=0$ for $1 \leq i \leq D$. Note that $a_1 \not=0$ implies
$a_i \not=0$ for $1 \leq i \leq D$.

\bigskip

\noindent Kazumasa Nomura \hfil\break
\noindent Tokyo Medical and Dental University \hfil\break
\noindent Kohnodai Ichikawa 272-0827 Japan \hfil\break
\noindent email: {\tt knomura@pop11.odn.ne.jp} \hfil\break

\noindent Paul Terwilliger \hfil\break
\noindent Department of Mathematics \hfil\break
\noindent University of Wisconsin \hfil\break
\noindent 480 Lincoln Drive \hfil\break
\noindent Madison, WI 53706-1388 USA \hfil\break
\noindent email: {\tt terwilli@math.wisc.edu }\hfil\break

\section{Statements and Declarations}

\noindent {\bf Funding}: The author declares that no funds, grants, or other support were received during the preparation of this manuscript.
\medskip

\noindent  {\bf Competing interests}:  The author  has no relevant financial or non-financial interests to disclose.
\medskip

\noindent {\bf Data availability}: All data generated or analyzed during this study are included in this published article.


\begin{thebibliography}{10}

\bibitem{BB}
E.~Bannai and Et.~Bannai.
\newblock Spin models on finite cyclic groups.
\newblock{\em J. Algebraic Combin.} 3 (1994) 243--259.

\bibitem{bbit}
E.~Bannai, Et.~Bannai, T.~Ito, R.~Tanaka.
\newblock
{\em Algebraic Combinatorics.}
\newblock De Gruyter Series in Discrete Math and Applications 5.
De Gruyter, 2021.      \\
https://doi.org/10.1515/9783110630251
 

\bibitem{bannai}
E.~Bannai and T.~Ito.
\newblock
{\em Algebraic Combinatorics, I. Association schemes.}
\newblock Benjamin/Cummings, Menlo Park, CA, 1984.



 
 
\bibitem{bcn}
A.~E.~Brouwer,  A.~Cohen, A.~Neumaier.
\newblock{\em Distance Regular-Graphs.}
\newblock Springer-Verlag, Berlin, 1989.


   


	


\bibitem{CW}
J.~S.~Caughman IV and
N.~Wolff.
\newblock
The Terwilliger algebra of a distance-regular graph that supports a spin model.
\newblock{\em J.~Algebraic Combin.} 21 (2005) 289--310.
		

\bibitem{cerzo}
 D.~Cerzo.
 \newblock Structure of thin irreducible modules of a $Q$-polynomial distance-regular graph.
 \newblock{\em Linear Algebra Appl.}  433  (2010)  1573--1613; {\tt arXiv:1003.5368}.

\bibitem{vinet1}
N.~Cramp\' e, F.~Nicolas,  J.~Luc Gaboriaud,  L.~Poulain  d'Andecy,    E.~ Ragoucy, L.~Vinet.
\newblock The Askey-Wilson algebra and its avatars.
\newblock{\em J. Phys. A } 54 (2021)  no. 6, Paper No. 063001, 32 pp.;
{\tt arXiv:2009.14815}.

\bibitem{vinet2}
N.~Cramp\'e, L.~Vinet, M.~Zaimi.
\newblock Braid group and $q$-Racah polynomials.
\newblock{\em Proc. Amer. Math. Soc.} 150 (2022) 951--966; {\tt arXiv:2106.02416}.


\bibitem{vinet3} 
N.~Cramp{\' e}, L.~Poulain d'Andecy, L.~Vinet, M.~Zaimi.
\newblock Askey-Wilson braid algebra and centralizer of $U_q(\mathfrak{sl}_2)$.
\newblock{\em Ann. Henri Poincar{\' e} }24 (2023) 1897--1922.
https://doi.org/10.1007/s00023-023-01275-4




\bibitem{curtin2hom}
B.~Curtin.
\newblock 2-Homogeneous bipartite distance-regular graphs.
\newblock{\em Discrete Math} 187 (1998) 39--70.

\bibitem{C:thin}
B.~Curtin.
\newblock Distance-regular graphs which support a spin model are thin.
\newblock{\em Discrete Math.} 197/198 (1999) 205--216.


\bibitem{C:spinLP}
B.~Curtin.
\newblock Spin Leonard pairs.
\newblock{\em Ramanujan J.} 13 (2007) 319--332.

\bibitem{mlt}
B.~Curtin.
\newblock
 Modular Leonard triples.
\newblock{\em Linear Algebra Appl.} 424 (2007)  510--539.


\bibitem{curtNom}
B.~Curtin and K.~Nomura.
\newblock
Some formulas for spin models on distance-regular graphs.
\newblock{\em
J. Combin. Theory Ser. B} 75 (1999) 206--236.


\bibitem{CNhom}
B.~Curtin and K.~Nomura.
\newblock Homogeneity of a distance-regular graph which supports a spin model.
\newblock {\em J. Algebraic Combin.} 19 (2004)  257--272.



\bibitem{dkt}   
E.~R.~van Dam, J.~H.~Koolen, H.~Tanaka.
\newblock Distance-regular graphs.
\newblock{\em Electron. J. Combin.} (2016) DS22;
{\tt arXiv:1410.6294}.


\bibitem{delsarte}
P.~Delsarte.
\newblock
An algebraic approach to the association schemes of coding theory.
\newblock
{\em Philips Research Reports Suppl.} 10  (1973).

\bibitem{fairlie}
D.~B.~Fairlie.
\newblock Quantum deformations of {SU}(2).
\newblock {\em J. Phys. A: Math. Gen.}   23  (1990) L183--L187.













	
	
		


\bibitem{havlicek}
M.~Havlí\v cek,  S.~Po\v sta, A.~U.~Klimyk.
\newblock
Representations of the cyclically
symmetric $q$-deformed algebra $U_q({\rm so}_3)$.   
\newblock{\em
Czechoslovak J. Phys.} 48 (1998)  1347--1353.


\bibitem{HS}
D.G.~Higman,
C.C.~Sims.
\newblock A simple group of order 44,352,000.
\newblock {\em Math. Zeitschr.} 105 (1968) 110--113.


\bibitem{hwh}
H.-W.~Huang.
\newblock
 The classification of Leonard triples of QRacah type.
\newblock{\em 
Linear Algebra Appl. } 436 (2012) 1442--1472; {\tt arXiv:1108.0458}.

\bibitem{hwh2}
H.-W.~Huang.
\newblock Finite-dimensional irreducible modules of the universal Askey-Wilson algebra.
\newblock{\em 
Comm. Math. Phys.} 340 (2015)  959--984; {\tt arXiv:1210.1740}.

\bibitem{hwh3}
H.-W.~Huang.
\newblock An embedding of the universal Askey-Wilson algebra into
$U_q(\mathfrak{sl}_2) \otimes U_q(\mathfrak{sl}_2) \otimes U_q(\mathfrak{sl}_2) $.
\newblock{\em 
Nuclear Phys. B} 922 (2017)  401--434.
https://doi.org/10.1016/j.nuclphysb.2017.07.007

\bibitem{someAlg}
T.~Ito, K.~Tanabe, P.~Terwilliger.
\newblock Some algebra related to $P$- and $Q$-polynomial association schemes.
\newblock{\em
Codes and Association Schemes (Piscataway NJ, 1999), 167--192, DIMACS
Ser. Discrete Math. Theoret. Comput. Sci.}
{56}, Amer. Math. Soc., Providence RI 2001;
{\tt arXiv:math.CO/0406556}.

\bibitem{qRacTet}
T.~Ito and P.~Terwilliger.
\newblock
 Distance-regular graphs of $q$-Racah type and the $q$-tetrahedron algebra.
\newblock{\em 
Michigan Math. J.} 58 (2009)  241--254; {\tt arXiv:0708.1992}.

\bibitem{dahaZ3}
T.~Ito and P.~Terwilliger.
\newblock Double affine Hecke algebras of rank 1 and the ${\mathbb Z}_3$-symmetric Askey-Wilson relations.
\newblock{\em 
SIGMA Symmetry Integrability Geom. Methods Appl.} 6 (2010), Paper 065, 9 pp.
 https://doi.org/10.3842/SIGMA.2010.065

\bibitem{augIto}
T.~Ito and P.~Terwilliger.
\newblock
The augmented tridiagonal algebra.
\newblock{\em  Kyushu J. Math.} 64 (2010) 81--144; {\tt arXiv:0904.2889}.



 

\bibitem{JaegerHS}
F.~Jaeger.
\newblock
 Strongly regular graphs and spin models for the Kauffman polynomial.
 \newblock{\em 
Geom. Dedicata} 44 (1992) 23--52.




\bibitem{Jones}
V.F.R.~Jones.
\newblock On knot invariants related to some statistical mechanical models.
\newblock {\em Pacific J. Math.} 137 (1989) 311--336.


\bibitem{kawagoe}
K.~Kawagoe, A.~Munemasa, Y.~Watatani.
\newblock Generalized spin models.
\newblock {\em J. Knot Theory Ramifications} 3 (1994)  465--475.

 
\bibitem{mun}
A.~Munemasa.
\newblock Private communication to Kazumasa  Nomura, 1994. 

 
\bibitem{nomHomNP}
K.~Nomura. 
\newblock Homogeneous graphs and regular near polygons.
\newblock{\em 
J. Combin. Theory Ser. B} 60 (1994)  63--71.


 \bibitem{nomHadamard}
 K.~Nomura.
 \newblock
 Spin models constructed from Hadamard matrices
 \newblock{\em 
J. Combin. Theory Ser. A} 68 (1994) 251--261.
 
 \bibitem{nomSM}
 K.~Nomura.
 \newblock
Spin models on bipartite distance-regular graphs.
\newblock{\em J. Combin. Theory
Ser. B} 64 (1995)  300--313.
 \bibitem{N:algebra}
K.~Nomura.
\newblock An algebra associated with a spin model.
\newblock{\em J. Algebraic  Combin.} 6 (1997) 53--58.
 
 

\bibitem{nomTB}
K.~ Nomura and P.~Terwilliger.
\newblock Totally bipartite tridiagonal pairs.
 \newblock {\em Electron. J. Linear Algebra } 37  (2021) 434--491; {\tt arXiv:1711.00332}.
 
 \bibitem{nomSpinModel}
 K.~ Nomura and P.~Terwilliger.
\newblock Leonard pairs, spin models, and distance-regular graphs.
\newblock{\em J. Combin. Theory Ser. A} 177 (2021), Paper No. 105312, 59 pp.;  {\tt arXiv:1907.03900}.
 

 


\bibitem{odesskii}
M.~Odesskii.
\newblock
An analogue of the Sklyanin algebra.
\newblock {\em  Funct. Anal. Appl.} 20 (1986)  152--154.   
 
 

 
\bibitem{sum2}
S.~Sumalroj.
\newblock 
A diagram associated with the subconstituent algebra of a distance-regular graph.
 \newblock {\em Ars Math. Contemp. } 17  (2019) 185--202; {\tt arXiv:1712.07692}.
 

 
 
 
 
\bibitem{tSub1} 
P.~Terwilliger. 
\newblock The subconstituent algebra of
an association scheme I.
\newblock{\em
J. Algebraic Combin.}
{ 1} (1992), 363--388.  

\bibitem{tSub2} 
P.~Terwilliger. 
\newblock The subconstituent algebra of
an association scheme II.
\newblock{\em
J. Algebraic Combin.}
{ 2} (1993), 73--103.  

\bibitem{tSub3} 
P.~Terwilliger. 
\newblock The subconstituent algebra of
an association scheme III. 
\newblock{ \em
J. Algebraic Combin.}
{ 2} (1993), 177--210.  

 
 

 



 \bibitem{qSerre}
  P.~Terwilliger.
   \newblock Two relations that generalize the $q$-Serre relations and the
      Dolan-Grady relations. In
        \newblock {\em  Physics and
	  Combinatorics 1999 (Nagoya)}, 377--398, World Scientific Publishing,
	    River Edge, NJ, 2001;
	     {\tt arXiv:math.QA/0307016}.


\bibitem{uaw}
P.~Terwilliger.
\newblock The universal Askey-Wilson algebra.
\newblock{\em
SIGMA Symmetry Integrability Geom. Methods Appl.} 7 (2011)  Paper 069, 24 pp.
 https://doi.org/10.3842/SIGMA.2011.069
 
\bibitem{uawsl2}
P.~Terwilliger.
\newblock The universal Askey-Wilson algebra and the equitable presentation of $U_q(\mathfrak{sl}_2)$.
\newblock{\em SIGMA Symmetry Integrability Geom. Methods Appl.} 7 (2011) Paper 099, 26 pp.; {\tt arXiv:1107.3544}.

\bibitem{uawDAHA}
P.~Terwilliger.
\newblock
The universal Askey-Wilson algebra and DAHA of type $(C^\vee_1, C_1)$.
\newblock{\em 
SIGMA Symmetry Integrability Geom. Methods Appl.} 9 (2013) Paper 047, 40 pp.
 https://doi.org/10.3842/SIGMA.2013.047

\bibitem{pseudo}
P.~Terwilliger.
\newblock Leonard triples of $q$-Racah type and their pseudo intertwiners.
\newblock{\em 
Linear Algebra Appl.} 515 (2017) 145--174; {\tt arXiv:1609.05488}.



\bibitem{TLus}
P.~Terwilliger.
\newblock
The Lusztig automorphism of the $q$-Onsager algebra.
\newblock{\em 
J. Algebra} 506 (2018)  56--75; {\tt arXiv:1706.05546}.

\bibitem{Tqtet}
P.~Terwilliger.
\newblock
Tridiagonal pairs of $q$-Racah type and the $q$-tetrahedron algebra.
\newblock{\em J. Pure Appl. Algebra} 225 (2021)  no. 8, Paper No. 106632, 33 pp.; {\tt arXiv:2006.02511}.

\bibitem{nortonPT}
P.~Terwilliger.
\newblock The Norton algebra of a $Q$-polynomial distance-regular graph.
\newblock{\em  J. Combin. Theory Ser. A } 182  (2021) Paper No. 105477, 11 pp.; {\tt arXiv:2006.04997}.

\bibitem{TLus2}
P.~Terwilliger.
\newblock
Twisting finite-dimensional modules for the $q$-Onsager algebra $\mathcal O_q$ via the Lusztig automorphism.
\newblock{\em Ramanujan J. } 61 (2023)  175--202; {\tt arXiv:2005.00457}.


\bibitem{int}
P.~Terwilliger.
\newblock
Distance-regular graphs, the subconstituent algebra, 
and the $Q$-polynomial property.
Preprint; {\tt arXiv:2207.07747}.
 

\bibitem{vidunas}
P.~Terwilliger and R.~Vidunas.
\newblock Leonard pairs and the Askey-Wilson relations.
 \newblock {\em J. Algebra Appl. } 3  (2004) 411--426; {\tt  arXiv:math/0305356}.

\bibitem{TZ}
P.~Terwilliger and A.~Zitnik.
\newblock
Distance-regular graphs of $q$-Racah type and the universal Askey-Wilson algebra.
\newblock{\em J.  Combin.Theory Ser. A}. 125 (2014) 98--112; {\tt arXiv:1307.7968}.



\bibitem{Zhidd}
A.~S. Zhedanov.
\newblock ``{H}idden symmetry'' of the {A}skey-{W}ilson polynomials,
\newblock {\em Teoret. Mat. Fiz.}   89 (1991) 190--204 (in Russian);
translation in 
{\em Theoret. Math. Phys.}  89 (1991) 1146--1157. 
			

\end{thebibliography}
\end{document}